\documentclass[10pt]{article}

\usepackage{latexsym,amsfonts,amsmath,amssymb,amsthm,url}
\usepackage[english]{babel}
\usepackage[latin1]{inputenc}
\usepackage[T1]{fontenc}
\usepackage{graphicx}
\usepackage{color}
\usepackage{dsfont}

\usepackage{filecontents}
\usepackage{pstricks}
\usepackage{pstricks-add}

\begin{document}

\newtheorem{duge}{Lemma}[section]
\newtheorem{prop}[duge]{Proposition}
\newtheorem{theo}[duge]{Theorem}
\newtheorem{cor}[duge]{Corollary}
\newtheorem{hypo}[duge]{Assumption}
\newtheorem{defi}[duge]{Definition}

\theoremstyle{definition}
\newtheorem{rem}[duge]{Remark}
\newtheorem*{ack}{Acknowledgments}

\renewcommand{\P}{\mathbf{P}}
\newcommand{\E}{\mathbf{E}}
\newcommand{\N}{\mathds{N}}
\newcommand{\Z}{\mathds{Z}}
\newcommand{\R}{\mathds{R}}
\newcommand{\T}{\mathds{T}}

\newcommand{\cste}[1]{\ensuremath{c_{#1}}}
\newcommand{\un}{\mathds{1}}
\newcommand{\defeq}{\ensuremath{\overset{\textup{\tiny{def}}}{=}}}
\newcommand{\ent}[1]{\lfloor{#1}\rfloor}

\newcommand{\pere}[1]{\ensuremath{\stackrel{\leftarrow}{#1}}}
\newcommand{\pereex}[2]{\raisebox{1pt}{\ensuremath{\stackrel{\leftarrow}{#1}_{\hspace*{-1pt}\raisebox{-1pt}{$\scriptstyle{#2}$}}}}}
\newcommand{\fils}[2]{\ensuremath{\stackrel{\rightarrow}{#1}^{\raisebox{-1pt}{$\hspace{0.5pt}\scriptscriptstyle{#2}$}}}}
\newcommand{\filsex}[3]{\raisebox{1pt}{\ensuremath{\stackrel{\rightarrow}{#1}_{\hspace*{-1pt}#2}^{\raisebox{-1pt}{$\hspace{0.5pt}\scriptscriptstyle{#3}$}}}}}

\newcommand{\sumstack}[2]{\ensuremath{\sum_{\substack{#1 \\ #2}}}}
\newcommand{\sumst}[3]{\ensuremath{\sum_{\substack{#1 \\ #2 \\ #3}}}}
\newcommand{\Pf}{\mathbf{P}}
\newcommand{\Ef}{\mathbf{E}}

\newcommand{\tra}[1]{\,{\vphantom{#1}}^t{#1}}

\title{\textbf{Recurrence and transience of a multi-excited random walk on a regular tree}}

\author{
\normalsize{\textsc{Anne-Laure Basdevant}\footnote{Institut de
Mathématiques de Toulouse, France.} \;and \textsc{Arvind
Singh}\footnote{Institut für Mathematik, Universität Zürich,
Schweiz.}\mbox{ }\footnote{Research supported by the Swiss Science
Foundation, grant PDAM2-114536/1.}}}
\date{}
\maketitle \vspace{-1cm}

\vspace*{0.2cm}

\begin{abstract} We study a model of multi-excited random
walk on a regular tree which generalizes the models of the once
excited random walk and the digging random walk introduced by Volkov
(2003). We show the existence of a phase transition and provide a
criterion for the recurrence/transience property of the walk. In
particular, we prove that the asymptotic behaviour of the walk
depends on the order of the excitations, which contrasts with the
one dimensional setting studied by Zerner (2005). We also consider
the limiting speed of the walk in the transient regime and
conjecture that it is not a monotonic function of the environment.
\end{abstract}

\bigskip
{\small{

\bigskip
\noindent{\bf A.M.S. Classification. } 60F20, 60K35, 60J80

\bigskip
\noindent{\bf e-mail. } anne-laure.basdevant@math.univ-toulouse.fr,
arvind.singh@math.unizh.ch}}

\section{Introduction}

The model of the excited random walk on $\Z^d$ was introduced by
Benjamini and Wilson in \cite{BenjaminiWilson03} and studied in
details in, for instance,
\cite{AmirBenjaminiKozma05,BerardRamirez07,Kozma03-Preprint,Kozma05-Preprint,VanderhofstadHolmes07,VanderhofstadHolmes08}.
Roughly speaking, it describes a walk which receives a push in some
specific direction each time it reaches a new vertex of $\Z^d$. Such
a walk is recurrent for $d=1$ and transient with linear speed for
$d\geq2$. In \cite{Zerner05,Zerner06}, Zerner introduced a generalization of
this model called multi-excited random walk (or cookie random walk)
where the walk receives a push, not only on its
first visit to a site, but also on some subsequent visits. This
model has received particular attention in the one-dimensional
setting (\emph{c.f.}
\cite{AntalRedner05,BasdevantSingh08a,BasdevantSingh08b,KosyginaZerner08,MountfordPimentelValle06}
and the references therein) and is relatively well understood. In
particular, a one-dimensional multi-excited random walk can be
recurrent or transient depending on the strength of the excitations
and may exhibit sub-linear growth in the transient regime.

Concerning multi-excited random walks in higher dimensions, not much
is known when one allows the excitations provided to the walk to
point in different directions. For instance, as remarked in
\cite{KosyginaZerner08}, for $d\geq 2$, when the
excitations of a $2$-cookies random walk push the walk in opposite
directions, then there is, so far, no known criterion for the
direction of transience. In this paper, we consider a similar model
where the state space of the walk is a regular tree and we allow the
excitations to point in opposite directions. Even in this
setting simpler than $\Z^d$, the walk exhibits a complicated phase
transition concerning its recurrence/transience behaviour.
\medskip

Let us be a bit more precise about the model. We consider a rooted
$b$-ary tree $\T$. At each vertex of the tree, we initially put a
pile of $M\geq 1$ "cookies" with ordered strengths $p_1,\ldots,p_M
\in [0,1)$. Let us also choose some other parameter $q\in(0,1)$
representing the bias of the walk after excitation. Then, a cookie
random walk on $\T$ is a nearest neighbor random walk $X =
(X_n)_{n\geq0}$, starting from the root of the tree and moving
according to the following rules:
\begin{itemize}
\item If $X_n = x$ and there remain the cookies with strengths
$p_j,p_{j+1},\ldots,p_M$ at this vertex, then $X$ eats the cookie
with attached strength $p_j$ and then jumps at time $n+1$ to the
father of $x$ with probability $1-p_j$ and to each son of $x$ with
probability $p_j/b$.
\item If $X_n = x$ and there is no remaining cookie at site $x$, then $X$
jumps at time $n+1$ to the father of $x$  with probability $1-q$ and
to each son of $x$ with probability $q/b$.
\end{itemize}

In particular, the bias provided to the height process $|X|$ upon
consuming a cookie of strength $p$ is $2p-1$. Therefore, a cookie
pushes the walk toward the root when $p<1/2$ and towards infinity
when $p>1/2$. The main question we address in this paper is to
investigate, whether $X$ is recurrent or transient \emph{i.e.} does
it return infinitely often to the origin or does it wander to
infinity.

For the one dimensional cookie random walk, a remarkably simple
criterion for the recurrence of the walk was obtained by Zerner
\cite{Zerner05} and generalized in \cite{KosyginaZerner08}.
This characterization shows that the behavior of the walks depends
only on the sum of the strengths of the cookies, but not on their
respective positions in the pile. However, in the tree setting
considered here, as in the multi-dimensional setting, the order of
the cookies does matter, meaning that inverting the position of two
cookies in the pile may affect the asymptotic behaviour of the walk.
We give here a criterion for recurrence from which we derive
explicit formulas for particular types of cookie environments.

\subsection{The model}

Let us now give a rigorous definition of the transition
probabilities of the walk and set some notations. In the remainder
of this paper, $\T$ will always denote a rooted $b$-ary tree with
$b\geq 2$. The root of the tree is denoted by $o$. Given $x\in\T$,
let $\pere{x}$ stand for the father of $x$ and
$\fils{x}{1},\fils{x}{2},\ldots,\fils{x}{b}$ stand for the sons of
$x$. We also use the notation $|x|$ to denote the height of a vertex
$x\in\T$. For convenience, we also add an additional edge from the
root to itself and adopt the convention that the father of the root
is the root itself ($\pere{o} = o$).

We call  cookie environment a vector $\mathcal{C} =
(p_1,p_2,\ldots,p_M\,;q) \in [0,1)^{M}\times (0,1)$, where $M\geq 1$
is the number of cookies. We put a semicolon before the last
component of the vector to emphasize  the particular role played by
$q$. A $\mathcal{C}$ multi-excited (or cookie) random walk is a
stochastic process $X = (X_n)_{n\geq 0}$ defined on some probability
space $(\Omega,\mathcal{F},\P)$, taking values in $\T$ with
transition probabilities given by
\begin{equation*}
\begin{aligned}
&\P\big\{ X_{0} = o \big\} = 1,\\
&\P\big\{ X_{n+1} = \filsex{X}{n}{i} \;|\;X_0,\ldots,X_n \big\} =
\left\{
\begin{array}{ll}
\frac{p_{j}}{b}&\hbox{if $j \leq M$,}\\
\frac{q}{b}&\hbox{if $j > M$,}\\
\end{array}
\right.\\
&\P\big\{ X_{n+1} = \pereex{X}{n} \;|\;X_0,\ldots,X_n \big\} =
\left\{
\begin{array}{ll} 1-p_{j} &\hbox{if $j \leq M$,}\\
1- q&\hbox{if $j > M$,}\\
\end{array}
\right.
\end{aligned}
\end{equation*}
where $i\in \{1,\ldots,b\}$ and $j\defeq\sharp\{0\leq k\leq n,
X_k=X_n\}$ is the number of previous visits of the walk to its
present position.

\begin{rem}
\begin{enumerate}
\item We do not allow $q=0$ in the definition of a cookie environment.
This assumption is made to insure that a $0-1$ law holds for the  walk.
 Yet, the method developed in this paper also enables to
treat the case $q=0$, \textit{c.f.} Remark \ref{sectqzzero}.
\item When $p_1 = p_2 = \ldots = p_M = q$, then $X$ is a classical random
walk on $\T$ and its height process is a drifted random walk on
$\Z$. Therefore, the walk is recurrent for $q \leq \frac{1}{2}$ and
transient for $q> \frac{1}{2}$. More generally, an easy coupling
argument shows that, when all the $p_i$'s and $q$ are smaller than
$\frac{1}{2}$ (resp. larger than $\frac{1}{2}$), the walk is
recurrent (resp. transient). The interesting cases occur when at
least one of the cookies pushes the walk in a direction opposite to
the bias $q$ of the walk after excitation.
\item This model was previously considered by Volkov \cite{Volkov03}
for the particular cookie environments:
\begin{enumerate}
\item $(p_1\,;\frac{b}{b+1})$ "once-excited random walk".
\item $(0,0\,;\frac{b}{b+1})$ "two-digging random walk".
\end{enumerate}
In both cases, Volkov proved that the walk is transient with linear
speed and conjectured that, more generally, any cookie random walk
which moves, after excitation, like a simple random walk on the tree
(\emph{i.e.} $q=b/(b+1)$) is transient. Theorem \ref{MainTheo} below
shows that such is indeed the case.
\end{enumerate}
\end{rem}

\begin{theo}[{\bf Recurrence/Transience criterion}] \label{MainTheo} \mbox{ }\\
Let $\mathcal{C} = (p_1,p_2, \ldots,p_M\,;q)$ be a cookie
environment and let $P(\mathcal{C})$ denote its associated cookie
environment matrix as in Definition \ref{defxip}. This matrix has
only a finite number of irreducible classes. Let
$\lambda(\mathcal{C})$ denote the largest spectral radius of theses
irreducible sub-matrices (in the sense of Definition \ref{defSpec}).
\begin{itemize}
\item[\textup{(a)}] If $q<\frac{b}{b+1}$ and $\lambda(\mathcal{C}) \leq \frac{1}{b}$, then the walk in the
cookie environment $\mathcal{C}$ is recurrent \emph{i.e.} it hits
any vertex of $\T$ infinitely often with probability $1$. Furthermore, if
$\lambda(\mathcal{C}) < \frac{1}{b}$, then the walk is positive
recurrent \emph{i.e.} all the return times to the root have finite
expectation.
\item[\textup{(b)}] If $q\geq\frac{b}{b+1}$ or $\lambda(\mathcal{C}) > \frac{1}{b}$, then the walk is
transient \emph{i.e.} $\lim_{n\to\infty}|X_n| = +\infty$.
\end{itemize}
Moreover, if $\tilde{\mathcal{C}} =
(\tilde{p}_1,\tilde{p_2},\ldots,\tilde{p}_M\,;\tilde{q})$ denotes
another cookie environment such that $\mathcal{C}\leq
\tilde{\mathcal{C}}$ for the canonical partial order, then the
$\tilde{\mathcal{C}}$ cookie random walk is transient whenever the
$\mathcal{C}$ cookie random walk is transient. Conversely, if the
$\tilde{\mathcal{C}}$ cookie random walk is recurrent, then so is
the $\mathcal{C}$ cookie random walk.
\end{theo}

The matrix $P(\mathcal{C})$ of the theorem is explicit. Its
coefficients can be expressed as a rational function of the $p_i$'s
and $q$ and its irreducible classes are described in Section \ref{sectionirredclasses}.
However, we do not know, except in particular cases, a
simple formula for the spectral radius $\lambda(\mathcal{C})$.

\medskip

Let us stress that the condition $\lambda(\mathcal{C}) \leq
\frac{1}{b}$ does not, by itself, insure the recurrence of the walk.
Indeed, when $X$ a biased random walk on the tree ($p_1=\ldots=p_M =
q$), then $P(\mathcal{C})$ is the transition matrix of a Galton-Watson
process with geometric reproduction law with parameter
$\frac{q}{q+b(1-q)}$. According to \cite{SenetaVereJones66}, we have
\begin{equation*}
\lambda(\mathcal{C)} = \left\{
\begin{array}{ll}
\frac{q}{b(1-q)} & \hbox{for } q\leq \frac{b}{b+1},\\
\frac{b(1-q)}{q} & \hbox{for } q > \frac{b}{b+1}.
\end{array}
\right.
\end{equation*}
Therefore, for $q$ sufficiently close to $1$, the walk is transient
yet $\lambda(\mathcal{C)} < 1/b$.

\medskip

Let us also remark that the monotonicity property of the walk with
respect to the initial cookie environment stated in Theorem
\ref{MainTheo}, although being quite natural, is not straightforward
since there is no simple way to couple two walks with different
cookie environments (in fact, we suspect that such a coupling does
not exist in general, see the conjecture concerning the monotonicity
of the speed below).

\begin{theo}[{\bf Speed and CLT when $\mathbf{p_i>0}$}]\label{TheoSpeed}\mbox{ }\\
Let $\mathcal{C} = (p_1,p_2,\ldots,p_M\,;q)$ be a cookie environment
such that $p_i>0$ for all $i$. If the $\mathcal{C}$-cookie random
walk is transient, then it has a positive speed and a central limit
theorem holds: there exist deterministic $v = v(\mathcal{C}) >0$ and
$\sigma =\sigma(\mathcal{C}) >0$ such that
$$
\frac{|X_n|}{n}
\overset{\hbox{\tiny{a.s.}}}{\underset{n\to\infty}{\longrightarrow}}
v \quad \hbox{ and } \quad  \frac{|X_n| - nv}{\sqrt{n}}
\overset{\hbox{\tiny{law}}}{\underset{n\to\infty}{\longrightarrow}}
\mathcal{N}(0,\sigma^2).
$$
\end{theo}

The assumption that all cookies have positive strength cannot be
removed. When some cookies have zero strength, it is possible to
construct a transient walk with sub-linear growth, \textit{c.f.}
Proposition \ref{PropVit0}.

A natural question to address is the monotonicity of the speed. It
is known that the speed of a one-dimensional cookie random walk is
non decreasing with respect to the cookie environment. However,
numerical simulations suggest that such is not the case for
the model considered here (\textit{c.f.} Figure \ref{figspeed}). We
believe this behaviour to be somewhat similar to that observed for a
biased random walk on a Galton-Watson tree: the slowdown of the walk
is due to the creation of "traps" where the walk spends a long time.
When $p_2 = 0$, this is easily understood by the following heuristic
argument: the walk returns to each visited site at
least once (except on the boundary of its trace) and the length of an excursion of the walk away from the
set of vertices it has already visited is a geometric random
variable with parameter $p_1$ (the first time the walk moves a step
towards the root, it moves back all the way until it reaches a
vertex visited at least twice). Therefore, as $p_1$ increases to
$1$, the expectation of the length of theses excursions goes to
infinity so we can expect the speed of the walk to go to $0$. What
we find more surprising is that this slowdown also seems to hold
true, to some extend, when $p_2$ is not zero, contrarily to the conjecture that
the speed of a biased random walk on a Galton-Watson tree with no leaf is monotonic, \emph{c.f.} Question 2.1 of \cite{LyonsPemantlePeres97}.

\begin{figure}
\begin{center}
\psset{xunit=6cm, yunit=6cm}
\begin{pspicture}[](0,-0.05)(1,1)
    \psaxes[Dx=0.2,Dy=0.2,tickstyle=bottom]{->}(0,0)(0,0)(1.05,1.05)
    \uput[0](1.05,0){$p_1$}
    \uput[90](0,1.05){$v(p_1)$}
    \psline[linecolor=black,linestyle=dashed,linewidth=0.3pt](0,1)(1,1)
    \psline[linecolor=black,linestyle=dashed,linewidth=0.3pt](1,0)(1,1)
    \psline[linecolor=red,linewidth=1pt](0.99,0.33294)(1,1)
    \fileplot[linewidth=1pt,linecolor=red]{data1.dat}
\end{pspicture}
\caption{\label{figspeed}Speed of a $(p_1,0.01\,;0.95)$ cookie
random walk on a binary tree obtained by Monte Carlo simulation.}
\end{center}
\end{figure}
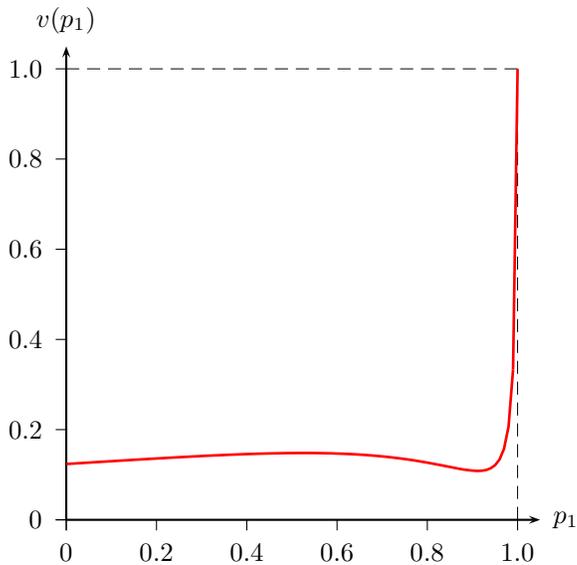

\subsection{Special cookie environments}

The value of the critical parameter $\lambda(\mathcal{C})$ can be
explicitly computed in some cases of interest.

\begin{theo}\label{TheoNcook}
Let $\mathcal{C} = (p_1,\ldots,p_M\,; q)$ denote a cookie
environment such that
\begin{equation}\label{hypo0}
p_i = 0 \quad\hbox{ for all $i \leq \lfloor M/2 \rfloor$}
\end{equation}
where $\lfloor x \rfloor$ denotes the integer part of $x$.  Define
\begin{equation*}
\lambda_{\hbox{\tiny{sym}}}(\mathcal{C}) \defeq \frac{q}{b(1-q)}
\prod_{i=1}^M\left((1-p_i)\left(\frac{q}{b(1-q)}\right)
+\frac{(b-1)p_i}{b}+\frac{p_i}{b}\left(\frac{q}{b(1-q)}\right)^{-1}\right).
\end{equation*}
For $q < \frac{b}{b+1}$, it holds that
\begin{equation*}
\lambda(\mathcal{C}) = \lambda_{\hbox{\tiny{sym}}}(\mathcal{C}).
\end{equation*}
\end{theo}

\begin{rem}\label{rem1}
For any cookie environment, we have $\lambda(\mathcal{C}) \leq 1$
(it is the maximal spectral radius of sub-stochastic matrices).
Moreover, when $\lfloor M/2 \rfloor$ cookies have strength $0$, the
function $q \mapsto \lambda_{\hbox{\tiny{sym}}}(p_1,\ldots,p_M\,;
q)$ is strictly increasing and
$\lambda_{\hbox{\tiny{sym}}}(p_1,\ldots,p_M\,;\frac{b}{b+1}) = 1$.
Thus, $\lambda(\mathcal{C}) \leq 1 <
\lambda_{\hbox{\tiny{sym}}}(\mathcal{C})$ for all $q >
\frac{b}{b+1}$.
\end{rem}

Let us also note that, under Assumption $(\ref{hypo0})$, in order to
reach some vertex $x$, the walk has to visit every vertex on the
path $[o,\pere{x})$ at least $M$ times. Therefore, for such a walk,
except on the boundary of its trace, every vertex of the tree is
visited either $0$ or more than $M$ times. This justifies
$\lambda(\mathcal{C})$ being, in this case, a symmetric function of the
$p_i$'s.

\medskip

The combination of Theorem \ref{MainTheo}, Theorem \ref{TheoNcook}
and Remark \ref{rem1} directly yields particularly simple criterions
for the model of the once excited and the digging random walk.

\begin{figure}
\begin{center}
\psset{xunit=6cm, yunit=6cm}
\begin{pspicture}[](0,-0.05)(1,1)
    \psaxes[Dx=0.5,Dy=0.5,tickstyle=bottom]{->}(0,0)(0,0)(1.05,1.05)
    \uput[0](1.05,0){$q$}
    \uput[90](0,1.05){$p$}
    \uput[90](0.61,-0.08){$\scriptstyle{2-\sqrt{2}}$}
    \psline[linecolor=black,linewidth=0.5pt](0,1)(1,1)
    \psline[linecolor=black,linewidth=0.5pt](1,0)(1,1)
    \psline[linecolor=black,linestyle=dashed,linewidth=0.4pt](0.5,0)(0.5,1)
    \psline[linecolor=black,linestyle=dashed,linewidth=0.4pt](0,0.5)(1,0.5)
    \uput[90](0.25,0.35){\large{Recurrence}}
    \uput[90](0.75,0.55){\large{Transience}}
    \psplot[linecolor=red, linewidth=1pt]{0}{0.5858}{
    4 x mul 2 sub x 2 exp sub
    3 x mul 2 sub
    div
    }
\end{pspicture}
\caption{\label{fig1exc}Phase transition of a $(p\,;q)$ cookie
random walk on a binary tree.}
\end{center}
\end{figure}
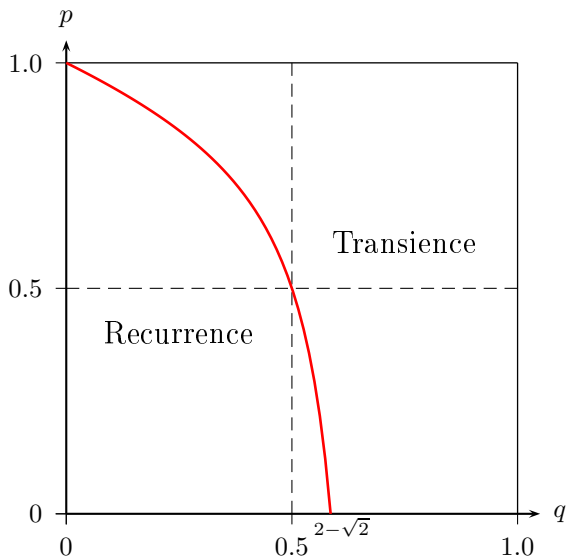

\begin{cor}[{\bf Once excited random walk}]\label{CorOnceExc}\mbox{ }\\
Let $X$ denote a $(p\,; q)$ cookie random walk (\textit{i.e.} $M=1$)
and define
\begin{equation*}
\lambda_{1} \defeq
(1-p)\left(\frac{q}{b(1-q)}\right)^2+\frac{(b-1)p}{b}\left(\frac{q}{b(1-q)}\right)+\frac{p}{b}.
\end{equation*}
Then $X$ is recurrent if and only if $\lambda_1 \leq 1/b$.
\end{cor}

In particular, the phase transition of the once excited random walk
is non trivial in both cases $p<\frac{1}{2}<q$ and $q<\frac{1}{2}<p$
(\textit{c.f.} Figure \ref{fig1exc}).

\medskip

\begin{cor}[{\bf $\mathbf{M}$-digging random walk}]\label{CorDigging}\mbox{ }\\
Let $X$ denote a $\mathcal{C} =
(\underbrace{0,\ldots,0}_{\hbox{\tiny{M times}}}\,; q)$ cookie
random walk and define
\begin{equation*}
\lambda_{\hbox{\tiny{dig}}} \defeq
\left(\frac{q}{b(1-q)}\right)^{M+1}.
\end{equation*}
Then $X$ is recurrent if and only if $\lambda_{\hbox{\tiny{dig}}}
\leq 1/b$.
\end{cor}

Recall that, according to  Theorem \ref{MainTheo}, the condition
$q\geq b/(b+1)$ is sufficient to insure the transience of the walk.
Corollary \ref{CorDigging} shows that this condition is also
necessary to insure transience independently of $p_1,\ldots,p_M$:
for any $q < b/(b+1)$, the $M$ digging random walk is recurrent when
$M$ is chosen large enough.

\bigskip

We now consider another class of cookie environment to show that,
contrarily to the one dimensional case, the order of the cookies in
the pile does matter in general.

\begin{prop}\label{Prop4Cook}
Let $X$ be a $\mathcal{C} =
(p_1,p_2,\underbrace{0,\ldots,0}_{\hbox{\tiny{K times}}}\,; q)$
cookie random walk with $K\geq 2$. Define $\nu(p_1,p_2)$ to be the
largest positive eigenvalue of the matrix
\begin{equation*}
\begin{pmatrix}
    \frac{p_1}{b} +\frac{p_1 p_2}{b} -\frac{2 p_1 p_2}{b^2} & \frac{p_1 p_2}{b^2} \\
    \frac{p_1 + p_2}{b} -\frac{2 p_1 p_2}{b^2} & \frac{p_1 p_2}{b^2}  \\
\end{pmatrix},
\end{equation*}
namely
\begin{equation*}
\nu(p_1,p_2)=\frac{1}{2b^2}\left((b\!-\!1)p_1p_2+bp_1+\sqrt{(b^2\!-\!6b\!+\!1)p_1^2p_2^2+2b(b\!-\!1)p_1^2p_2+b^2p_1^2+4bp_1p_2^2}\right).
\end{equation*}
Recall the definition of
$\lambda_{\hbox{\textup{\tiny{sym}}}}(\mathcal{C})$ given in Theorem
\ref{TheoNcook} and set
\begin{equation*} \tilde{\lambda} \;=\; \max\left(
\lambda_{\hbox{\textup{\tiny{sym}}}}(\mathcal{C}) ,
\nu(p_1,p_2)\right).
\end{equation*}
The walk $X$ is recurrent if and only if $\tilde{\lambda} \leq
\frac{1}{b}$.
\end{prop}

Since $\nu$ is not symmetric in $(p_1,p_2)$, Proposition
\ref{Prop4Cook} confirms that it is possible to construct a recurrent
cookie random walk such that the inversion of the first two cookies
yields a transient random walk. For $b=2$, one can choose, for
example, $p_1=\frac{1}{2}$, $p_2=\frac{4}{5}$ and $q\leq
\frac{1}{2}$.

Proposition \ref{Prop4Cook} also enables to construct a transient
cookie random walk with sub-linear growth.
\begin{prop}\label{PropVit0}
Let $X$ be a $\mathcal{C} = (p_1,p_2,0,0\,; q)$
cookie random walk with $q\geq b/(b+1)$ and $\nu(p_1,p_2) =
1/b$. Then $X$ is transient yet
\begin{equation*}
\liminf_{n\to\infty} \frac{|X_n|}{n} = 0.
\end{equation*}
\end{prop}
We do not know whether the liminf above is, in fact, a limit.

\bigskip

The remainder of this paper is organized as follows. In the next
section, we prove a $0-1$ law for the cookie random walk. In section
$3$, we introduce a branching Markov chain $L$ (or equivalently a
multi-type branching process with infinitely many types) associated
with the local time of the walk. We show that the walk is recurrent
if and only if this process dies out almost surely. We also prove
some monotonicity properties of the process $L$ which imply the
monotonicity property of the cookie random walk stated in Theorem
\ref{MainTheo}. In section $4$, we study the decomposition of the
transition matrix $P$ of $L$ and provide some results concerning the
evolution of a tagged particle. Section $5$ is devoted to completing
the proof of Theorem \ref{MainTheo}. In section $6$, we prove the
law of large number and C.L.T. of Theorem \ref{TheoSpeed} and
Proposition \ref{PropVit0}. In section \ref{sectioncalcul}, we compute the value of
the critical parameter $\lambda(\mathcal{C})$ for the special cookie
environments mentioned above and prove Theorem \ref{TheoNcook} and
Proposition \ref{Prop4Cook}. Finally, in the last section, we
discuss some possible extensions of the model.

\newpage

\section{The 0 - 1 law}\label{section0-1}

In the remainder of the paper, $X$ will always denote a $\mathcal{C}
= (p_1,\ldots,p_M\,; q)$ cookie random walk on a $b$-ary tree $\T$.
We denote by $\T^x$ the sub-tree of $\T$ rooted at $x$. For
$n\in\N$, we also use the notation $\T_{n}$ (resp. $\T_{\le n}$,
$\T_{<n}$) to denote the set of vertices which are at height $n$
(resp. at height $\leq n$ and $< n$) from the root . We introduce
the sequence $(\tau^k_{o})_{k\geq 0}$ of return times to the root.
\begin{equation*}
\left\{
\begin{aligned}
&\tau^0_{o} \;\defeq\; 0,\\
&\tau^{k+1}_{o}\; \defeq \; \min\{i>\tau^k_{o},\; X_i = o\},\\
\end{aligned}
\right.
\end{equation*}
with the convention $\min\{\emptyset\} = \infty$.  The following
result shows that, although a cookie random walk is not a Markov
process, a $0-1$ law holds (recall that we assume $q\neq 0$ in the
definition of a cookie environment).

\begin{duge}[{\bf $\mathbf{0}-\mathbf{1}$ law}]\label{loi0-1}
Let $X$ be a $\mathcal{C}$ cookie random walk.
\begin{enumerate}
\item If there exists $k\geq 1$ such that $\P\{\tau^k_o=\infty\}>0$, then $\lim_{n\rightarrow \infty}
|X_n|=\infty$ $\P$-a.s.
\item Otherwise, the walk visits any vertex infinitely often $\P$-a.s.
\end{enumerate}
\end{duge}

\begin{proof}
Let us first assume that $\P\{\tau_o^k<\infty\}=1$ for all $k$
\emph{i.e.} the walk returns infinitely often to the origin almost
surely. Since there are no cookies left after the
$M^{\hbox{\tiny{th}}}$ visit of the root, the walk will visit every
 vertex of height $1$ infinitely often with probability $1$. By
induction, we conclude that the walk visits every vertex of $\T$
infinitely often almost surely.

We now prove the transience part of the proposition. We assume that
$\P\{\tau_o^{k_0}<\infty\} < 1$ for some $k_0\in \N$. Let $\Omega_1$
denote the event
\begin{equation*}
\Omega_1\;\defeq\;\Big\{ \lim_{i\to \infty} |X_i| =\infty\Big\}^c.
\end{equation*}
Given $N\in \N$, let $\tilde{X}^N$ denote a multi-excited random
walk on $\T$ reflected at height $N$ (\emph{i.e.} a process with
the same transition rule as $X$ but which always goes back to its father
when it reaches a vertex of height $N$). This process takes values
in the finite state space $\T_{\leq N}$ and thus visits any site of
$\T_{\leq N}$ infinitely often almost surely. For $x\in \T_{< N}$,
let $\tilde{\tau}_x^{k_0}$ be the time of the
$k_0^{\hbox{\tiny{th}}}$ return of $\tilde{X}^N$ to the vertex $x$.
For $n<N$, let also $\tilde{\tau}^{k_0}_n = \sup_{x\in \T_{n}}
\tilde{\tau}_x^{k_0}$ be the first time when all the vertices of
height $n$ have been visited at least $k_0$ times. We consider the
family of events $(A_{n,N})_{n<N}$ defined by:
\begin{equation*}
A_{n,N}\;\defeq\;\{\tilde{X}^N \mbox{does not reach height $N$ before
$\tilde{\tau}^{k_0}_n$}\}.
\end{equation*}
Let us note that, on $A_{n,N}$, the processes $X$ and $\tilde{X}^N$
are equal up to time $\tilde{\tau}^{k_0}_n$. Moreover, given
$n\in\N$ and $\omega \in \Omega_1$, we can always find $N
> n$ such that $\omega \in A_{n,N}$. Hence,
\begin{equation*}
\Omega_1\subset \bigcap_{n\ge 1} \bigcup_{N>n} A_{n,N}.
\end{equation*}
In particular, for any fixed $n\geq 1$, we get
\begin{equation}\label{eq01-1}
\P\{\Omega_1\}\le \sup_{N>n} \P\{A_{n,N}\}.
\end{equation}
It remains to bound $\P\{A_{n,N}\}$. For $x \in \T_{n}$, we
consider the subsets of indices:
\begin{equation*}
\begin{aligned}
&I_x \;\defeq\; \{0\le i \le \tilde{\tau}^{k_0}_n,\; \tilde{X}^N_i \in \T^x\}.\\
&I'_x \;\defeq\; \{0\le i \le \tilde{\tau}_x^{k_0},\; \tilde{X}^N_i
\in \T^x\}\subset I_x.
\end{aligned}
\end{equation*}
With these notations, we have
\begin{eqnarray*}
\P\{A_{n,N}\}& = &\P\{\forall x \in \T_{n},\; (\tilde{X}^N_i,i\in I_x) \mbox{ does not reach height $N$}\}\\
& \le & \P\{\forall x \in \T_{n},\; (\tilde{X}^N_i,i\in I'_x)
\mbox{ does not reach height $N$}\}.
\end{eqnarray*}
Since the multi-excited random walk evolves independently in
distinct subtrees, up to a translation, the stochastic processes
$(\tilde{X}^N_i,i\in I'_x)_{x\in \T_{n}}$ are i.i.d. and have the law
of the multi-excited random walk $X$ starting from the root $o$,
reflected at height $N-n$ and killed at its $k_0^{\hbox{\tiny{th}}}$
return to the root. Thus,
\begin{equation}\label{eq01-2}
\P\{A_{n,N}\} \leq \P\Big\{ (\tilde{X}^{N-n}_i,i\le
\tilde{\tau}_o^{k_0}) \mbox{ does not reach height
$N-n$}\Big\}^{b^n} \leq \P\{ \tau_o^{k_0}<\infty\}^{b^n}.
\end{equation}
Putting (\ref{eq01-1}) and (\ref{eq01-2}) together, we conclude that
\begin{equation*}
\P\{\Omega_1\}\le \P\{ \tau_o^{k_0}<\infty\}^{b^n}
\end{equation*}
and we complete the proof of the lemma by letting $n$ tend to
infinity.
\end{proof}

\section{The branching Markov chain $L$}\label{sectionL}
\subsection{Construction of $L$}\label{secdefL}
In this section, we construct a branching Markov chain which
coincides with the local time process of the walk in the recurrent
setting and show that the survival of this process characterizes the
transience of the walk.

Recall that $\tilde{X}^N$ denotes the cookie random walk $X$
reflected at height $N$. Fix $k_0
>0$. Let $\sigma_{k_0}$ denote the time of the
$k_0^{\hbox{\tiny{th}}}$ crossing of the edge joining the root of
the tree to itself:
\begin{equation*}
\sigma_{k_0} \;\defeq\; \inf\Big\{ i> 0,\;
\sum_{j=1}^{i}\un_{\{\tilde{X}^N_j=\tilde{X}^N_{j-1}=o\}} = k_0
\Big\}.
\end{equation*}
Since the reflected walk $\tilde{X}^N$ returns to the root
infinitely often, we have $\sigma_{k_0}<\infty$ almost surely. Let
now $\ell(x)$ denote the number of jumps of $\tilde{X}^N$ from
$\pere{x}$ to $x$ before time $\sigma_{k_0}$ \textit{i.e.}
\begin{equation*}
\ell^N(x) \;\defeq\; \sharp\{0\leq i < \sigma_{k_0}, \;
\tilde{X}^N_i = \pere{x} \hbox{ and } \tilde{X}^N_{i+1} = x\}.
\qquad\hbox{for all $x\in\T_{\leq N}$}
\end{equation*}
We consider the $(N+1)$-step process $L^N =
(L^N_0,L^N_1,\ldots,L^N_{N})$ where
\begin{equation*}
L^N_n\;\defeq\; (\ell^N(x),\; x\in \T_{n})\in \N^{\T_{n}}.
\end{equation*}
Since the quantities $L^N$, $\ell^{N}$ depend on $k_0$, we should
rigourously write $L^{N,k_0}$, $\ell^{N,k_0}$. Similarly, we should
write $\sigma^N_{k_0}$ instead of $\sigma_{k_0}$. Yet, in the whole
paper, for the sake of clarity, as we try to keep the notations as
simple as possible, we only add a subscript to emphasize  the
dependency upon some parameter when we feel that it is really
necessary. In particular, the dependency upon the cookie environment
$\mathcal{C}$ is usually implicit.

The process $L^N$ is Markovian, in order to compute its transition
probabilities we need to introduce some notations which we will
extensively use in the rest of the paper.

\begin{defi}\label{defxip}\mbox{ }
\begin{itemize}
\item Given a cookie environment $\mathcal{C}=(p_1,\ldots,p_M\,;q)$,
we denote by $(\xi_i)_{i\ge 1}$ a sequence of independent random
variables taking values in $\{0,1,\ldots,b\}$, with distribution:
\begin{eqnarray*}
\P\{\xi_i=0\} &=& \left\{
\begin{array}{ll}
1-p_i&\hbox{if $i \leq M$,}\\
1-q&\hbox{if $i > M$,}\\
\end{array}
\right.\\
\P\{ \xi_i=1\} = \ldots = \P\{ \xi_i=b\} &=& \left\{
\begin{array}{ll} \frac{p_{i}}{b} &\hbox{if $i \leq M$,}\\
\frac{q}{b}&\hbox{if $i > M$.}\\
\end{array}
\right.
\end{eqnarray*}
We say that $\xi_i$ is a "failure" when $\xi_i = 0$.
\item We call "cookie environment matrix" the non-negative matrix $P = (p(i,j))_{i,j \geq
0}$ whose coefficients are given by $p(0,j) = \un_{\{j=0\}}$ and,
for $i\geq 1$,
\begin{equation*}
p(i,j) \defeq \P\Big\{ \sum_{k=1}^{\gamma_i} \un_{\{\xi_k = 1\}} =
j\Big\}\quad \hbox{where}\quad \gamma_i \defeq \inf\Big\{n,\;
\sum_{k=1}^{n} \un_{\{\xi_k = 0\}} = i\Big\}.
\end{equation*}
Thus, $p(i,j)$ is the probability that there are exactly $j$ random
variables taking value $1$ before the $i^{\hbox{\tiny{th}}}$ failure
in the sequence $(\xi_1,\xi_2,\ldots)$.
\end{itemize}
\end{defi}

The following lemma characterizes the law of $L^N$.
\begin{duge}\label{transL} The process $L^N=(L^N_0,L^N_1,\ldots,L^N_N)$ is a Markov process on $\bigcup_{n=1}^N\N^{\T_{n}}$.
Its transition probabilities can be described as follows:
\begin{itemize}
\item[\textup{(a)}] $L_0 = (k_0)$ \emph{i.e.} $\ell(o) =
k_0$.
\item[\textup{(b)}] For $1\leq n\leq N$ and $x_1,\ldots,x_k\in \T_{n}$
with distinct fathers, conditionally on $L^N_{n-1}$, the random
variables $\ell^N(x_1),\ldots,\ell^N(x_k)$ are independent.
\item[\textup{(c)}] For $x\in \T_{n}$ with children $\fils{x}{1},\ldots,\fils{x}{b}$,
the law of $\big(\ell^N(\fils{x}{1}),\ldots,
\ell^N(\fils{x}{b})\big)$, conditionally on $L^N_n$, depends only on
$\ell^N(x)$ and is given by:
\begin{eqnarray*}
&&\hspace{-0.8cm}\P\Big\{\ell^N(\fils{x}{1})=0,\ldots,\ell^N(\fils{x}{b})=0\;
\Big|\;
\ell^N(x) = 0\Big\} = 1\\
&&\hspace{-0.8cm}\P\Big\{\ell^N(\fils{x}{1})=j_1,\ldots,\ell^N(\fils{x}{b})=j_b\;
\Big|\; \ell^N(x) = j_0 > 0\Big\}
\\&&\hspace{1cm}=\P\Big\{\forall k\in [0,b],\; \sharp\{1\le i\le j_0+\ldots+j_b,\; \xi_i=k\} = j_k
\mbox{ and } \,\xi_{j_0+\ldots+j_b}=0\Big\}.
\end{eqnarray*}
In particular, conditionally on $\ell^N(x) = j_0$, the random
variable  $\ell^N(\fils{x}{k})$ is distributed as the number of
$\xi_i$'s taking value $k$ before the $j_0^{\hbox{\tiny{th}}}$
failure. By symmetry, this distribution does not depend on $k$ and,
with the notation of Definition \ref{defxip}, we have
\begin{equation*}
\P\Big\{\ell^N(\fils{x}{k})= j \; \Big|\; \ell^N(x) = j_0\Big\} =
p(j_0,j).
\end{equation*}
\end{itemize}
\end{duge}
\begin{proof} (a) is a direct consequence of the definition of
$\sigma_{k_0}$.  Let $x\in\T_{\leq N}$. Since the walk $\tilde{X}^N$
is at the root of the tree at times $0$ and $\sigma_{k_0}$, the
number of jumps $\ell^N(x)$ from $\pere{x}$ to $x$ is equal to the
number of jumps from $x$ to $\pere{x}$. Moreover, the walk can only
enter and leave the subtree $\T^x\cap \T_{\leq N}$ by crossing the
edge $(x,\pere{x})$. Therefore, conditionally on $\ell^N(x)$, the
families of random variables $(\ell^N(y),\; y\in \T^x\cap \T_{\leq
N})$ and $(\ell^N(y),\; y\in \T_{\leq N} \backslash \T^x)$ are
independent. This fact implies (b) and the Markov property of $L$.
Finally, (c) follows readily from the definition of the transition
probabilities of a cookie random walk and the construction of the
sequence $(\xi_i)_{i\ge 1}$ in terms of the same cookie environment.
\end{proof}

In view of the previous lemma, it is clear that for all $x\in
\T_{\leq N}$, the distribution of the random variables $\ell^N(x)$
does not, in fact, depend on $N$. More precisely, for all $N'
> N$, the $(N+1)$ first steps $(L_0^{N'},\ldots,L_N^{N'})$ of the process $L^{N'}$ have the same
distribution as $(L_0^{N},\ldots,L_N^{N})$. Therefore, we can consider a Markov
process $L$ on the state space $\bigcup_{n=1}^\infty\N^{\T_{n}}$:
\begin{equation*}
L = (L_{n},\; n\geq 0)\quad\hbox{with}\quad L_n= (\ell(x),\; x\in
\T_{n})\in \N^{\T_{n}}
\end{equation*}
where, for each $N$, the family $(\ell(x),\; x\in\T_{\leq N})$ is
distributed as $(\ell^N(x),\; x\in\T_{\leq N})$. We can interpret
$L$ as a branching Markov chain (or equivalently a multi-type
branching process with infinitely many types) where the particles
alive at time $n$ are indexed by the vertices of $\T_{n}$:
\begin{itemize}
\item The process starts at time $0$ with one particle $o$ located
at $\ell(o) = k_0$.
\item At time $n$, there are $b^n$ particles in the system indexed by $\T_n$.
The position (in $\N$) of a particle $x$ is $\ell(x)$.
\item At time $n+1$, each particle $x\in\T_{n}$ evolves independently: it splits into $b$
particles $\fils{x}{1},\ldots,\fils{x}{b}$. The positions
$\ell(\fils{x}{1}),\ldots,\ell(\fils{x}{b})$ of these new particles,
conditionally on $\ell(x)$, are given by the transition kernel
described in (c) of the previous lemma.
\end{itemize}
\begin{rem}\mbox{ }
\begin{itemize}
\item[\textup{(1)}] Changing the value of $k_0$ only affects the position
$\ell(o)$ of the initial particle but does not change the transition
probabilities of the Markov process $L$. Thus, we shall denote by $\P_{k}$ the probability where the process $L$
starts from one particle located at $\ell(o)=k$. The notation
$\E_{k}$ will be used for the expectation under $\P_k$.
\item[\textup{(2)}] The state $0$ is absorbing for the branching Markov chain $L$:
if a particle is at $0$, then all its descendants remain at $0$ (if
the walk never crosses an edge $(\pere{x},x)$, then, \emph{a
fortiori}, it never crosses any edge of the subtree $\T^x$).
\item[\textup{(3)}] Let us stress that, given $\ell(x)$, the positions of the $b$
children $\ell(\fils{x}{1}),\ldots,\ell(\fils{x}{b})$ are not
independent. However, for two distinct particles, the evolution of
their progeny is independent \emph{c.f.} (b) of Lemma \ref{transL}.
\item[\textup{(4)}] When the cookie random walk $X$ is
recurrent, the process $L$ coincides with the local time process of
the walk and one can directly construct $L$ from $X$ without
reflecting the walk at height $N$ and taking the limit. However,
when the walk is transient, one cannot directly construct $L$ with
$N=\infty$. In this case, the local time process of the walk,
stopped at its $k_0^{\hbox{\tiny{th}}}$ jump from the root to itself
(possibly $\infty$), is not a Markov process.
\end{itemize}
\end{rem}

Since $0$ is an absorbing state for the Markov process $L$, we say
that $L$ dies out when there exists a time such that all the
particles are at $0$. The following proposition characterizes the
transience of the cookie random walk in terms of the survival of $L$.
\begin{prop}\label{equivalence}
The cookie random walk is recurrent if  and only if, for any choice
of $k$, the process $L$, under $\P_k$ (\textit{i.e.} starting from
one particle located at $\ell(o)=k$), dies out almost surely.
\end{prop}

\begin{proof}
Let us assume that, for any $k$, the process $L$ starting from
$k$ dies out almost surely. Then, $k$ being fixed, we can find
$N$ large enough such that $L$ dies out before time $N$ with
probability $c$ arbitrarily close to $1$. Looking at the definition
of $L$, this means that the walk $X$ crosses at least $k$ times
the edge $(o,\pere{o})$ before reaching level $N$ with probability
$c$. Letting $c$ tend to $1$, we conclude that $X$ returns to the
root at least $k$ times almost surely. Thus, the walk is recurrent.

Conversely, if, for some $k$, the process $L$ starting from $k$
has probability $c>0$ never to die out, then the walk $X$ crosses
the edge $(o,\pere{o})$  less than $k$ times with probability $c$.
This implies that $X$ returns to the root only a finite number of times
with strictly positive probability. According to Lemma \ref{loi0-1},
the walk is transient.
\end{proof}

Recall that, in the definition of a cookie environment, we do not
allow the strengths of the cookies $p_i$ to be equal to $1$. This
assumption insures that, for a particle $x$ located at $\ell(x) >M$,
the distribution $(\ell(\fils{x}{1}),\ldots,\ell(\fils{x}{b}))$ of
the position of its $b$ children has a positive density everywhere
on $\N^b$. Indeed, for any $j_1,\ldots,j_n \in \N$, the probability
\begin{equation}\label{posprob}
\P\Big\{\ell(\fils{x}{1})=j_1,\ldots,\ell(\fils{x}{b})=j_b \;|\;
\ell(x) = i >M\Big\}
\end{equation}
is larger that the probability of the $i+j_1+\ldots+j_b$ first terms
of the sequence $(\xi_k)_{k\geq 1}$ being
\begin{equation*}
\underbrace{0,\ldots,0}_{\hbox{\tiny{$i\!-\!1$
times}}},\underbrace{1,\ldots,1}_{\hbox{\tiny{$j_1$ times}}},\ldots,
\underbrace{b,\ldots,b}_{\hbox{\tiny{$j_b$ times}}},0
\end{equation*}
which is non zero. Therefore, we get the simpler criterion:
\begin{cor}\label{coroequiv}
The cookie random walk is recurrent if  and only if $L$ under $P_{M+1}$ dies out almost surely.
\end{cor}

\subsection{Monotonicity property of $L$}

The particular structure of the transition probabilities of $L$ in
terms of successes and failures in the sequence $(\xi_k)$ yields
useful monotonicity properties for this process.

Given two branching Markov chains $L$ and $\tilde{L}$, we say that
$L$ is stochastically dominated by $\tilde{L}$ if we can construct
both processes on the same probability space in such way that
$$
\ell(x) \leq \tilde{\ell}(x)\quad\hbox{ for all $x\in\T$, almost surely.}
$$

\begin{prop}[\textbf{monotonicity w.r.t. the initial position}]\label{propmonoL1}
For any $0\leq i \leq j$, the process $L$ under $\P_i$ is stochastically dominated by $L$ under $\P_j$.
\end{prop}
\begin{proof}
Since each particle in $L$ reproduces independently, we just need to
prove that $L_1 = (\ell(\fils{o}{1}),\ldots,\ell(\fils{o}{b}))$
under $\P_i$ is stochastically dominated by $L_1$ under $\P_j$ and
the result will follows by induction. Recalling that, under $\P_i$
(resp. $\P_j$),  $\ell(\fils{o}{k})$ is given by the number of random
variables $\xi$ taking value $k$ before the $i^{\hbox{\tiny{th}}}$
failure (resp. $j^{\hbox{\tiny{th}}}$ failure) in the sequence
$(\xi_{n})$, we conclude that, when $i\leq j$, we can indeed create
such a coupling by using the same sequence $(\xi_n)$ for both
processes.
\end{proof}

\begin{prop}[\textbf{monotonicity w.r.t. the cookie environment}]\label{propmonoL2}\mbox{ }\\
Let $\mathcal{C} = (p_1, \ldots,p_M\,;q)$ and $\tilde{\mathcal{C}} =
(\tilde{p}_1,\ldots,\tilde{p}_M\,;\tilde{q})$ denote two cookies
environments such that $\mathcal{C}\leq \tilde{\mathcal{C}}$ for the
canonical partial order. Let $L$ (resp. $\tilde{L}$) denote the
branching Markov chain associated with the cookie environment
$\mathcal{C}$ (resp. $\tilde{\mathcal{C}}$). Then, for any $i\geq
0$, under $\P_i$, the process $\tilde{L}$ stochastically dominates
$L$.
\end{prop}
\begin{proof}
Keeping in mind Proposition \ref{propmonoL1} and using again an
induction argument, we just need to prove the result for the first
step of the process \emph{i.e.} prove that we can construct $L_1$
and $\tilde{L}_1$ such that, under $\P_i$,
\begin{equation}\label{amontrmono}
\ell(\fils{o}{k}) \leq \tilde{\ell}(\fils{o}{k})\quad\hbox{ for all $k\in\{1,\ldots,b\}$.}
\end{equation}
Let $(\xi_n)$ denote a sequence of random variables as in Definition
\ref{defxip} associated with the cookie environment $\mathcal{C}$.
Similarly, let $(\tilde{\xi}_n)$ denote a sequence associated with
$\tilde{\mathcal{C}}$. When $\mathcal{C}\leq \tilde{\mathcal{C}}$,
we have $\P\{\xi_n = 0\} \geq \P\{\tilde{\xi}_n = 0\}$ and
$\P\{\xi_n = k\} \leq \P\{\tilde{\xi}_n = k\}$ for all
$k\in\{1,\ldots,b\}$. Moreover, the random variables $(\xi_n)_{n\geq
1}$ (resp. $(\tilde{\xi}_n)_{n\geq 1}$) are independent. Thus, we
can construct the two sequences $(\xi_n)$ and $(\tilde{\xi}_n)$ on
the same probability space in such way that for all $n\geq 1$ and
all $k \in\{1,\ldots,b\}$,
\begin{eqnarray*}
\tilde{\xi}_n = 0 &\hbox{ implies } &\xi_n = 0,\\
\xi_n =  k &\hbox{ implies } &\tilde{\xi}_n = k.
\end{eqnarray*}
Defining now, for each $k$, the random variable $\ell(\fils{o}{k})$
(resp. $\tilde{\ell}(\fils{o}{k})$) to be the number of random
variables taking value $k$ in the sequence $(\xi_n)$ (resp.
$(\tilde{\xi}_n)$) before the $i^{\hbox{\tiny{th}}}$ failure, it is
clear that (\ref{amontrmono}) holds.
\end{proof}

The monotonicity of the recurrence/transience behaviour of the
cookie walk with respect to the initial cookie environment stated in
Theorem \ref{MainTheo} now follows directly from the combination of
Corollary \ref{coroequiv} and Proposition \ref{propmonoL2}:

\begin{cor}\label{coromono}
Let $\mathcal{C} = (p_1,p_2, \ldots,p_M\,;q)$ and
$\tilde{\mathcal{C}} =
(\tilde{p}_1,\tilde{p_2},\ldots,\tilde{p}_M\,;\tilde{q})$ denote two
cookie environments such that $\mathcal{C}\leq \tilde{\mathcal{C}}$.
The $\tilde{\mathcal{C}}$ cookie random walk is transient whenever
the $\mathcal{C}$ cookie random walk is transient. Conversely, if
the $\tilde{\mathcal{C}}$ cookie random walk is recurrent, then so
is the $\mathcal{C}$ cookie random walk.
\end{cor}

\section{The Matrix $P$ and the process $Z$}
\subsection{Irreducible classes of $P$}\label{sectionirredclasses}

The matrix $P$ plays a key role in the study of  $L$. Since we allow
the strength of a cookie to be zero, the transition matrix $P$ need
not be irreducible (a matrix is said to be irreducible if, for any
$i,j$, there exists $n$ such that $p^{(n)}(i,j)>0$, where
$p^{(n)}(i,j)$ denotes the $(i,j)$ coefficient of $P^n$).

For $i,j\in \N$, we use the classical notations

\begin{itemize}
\item  $i\rightarrow j\;$ if $p^{(n)}(i,j)>0$ for some $n\ge 1$.
\item $i\leftrightarrow j\;$ if $\;i\rightarrow j$ and $j\rightarrow i$.
\end{itemize}

\begin{duge}\label{lemmmono}
For any $i,j\in \N$, we have
\begin{itemize}
\item[\textup{(a)}]If $p(i,j)>0$ then $p(i,k)>0$ for all $k\le j$ and $p(k,j)>0$ for all $k\ge i$.
\item[\textup{(b)}] If $i\rightarrow j$ then $i\rightarrow k$ for all $k\le j$ and $k\rightarrow j$ for all $k\ge i$.
\end{itemize}
\end{duge}
\begin{proof}
Recall the specific form of the coefficients of $P$: $p(i,j)$ is the
probability of having $j$ times $1$ in the sequence $(\xi_n)_{n\ge
1}$ before the $i^{\hbox{\tiny{th}}}$ failure. Let us also note that
we can always transform a realization of $(\xi_n)_{n\ge 1}$
contributing to $p(i,j)$ into a realization contributing to $p(i,k)$
for $k\leq j$ (resp. for $p(k,j)$ for $k\geq i$) by inserting
additional failures in the sequence. Since no cookie has strength
$1$, for any $n\ge 1$, $\P\{\xi_n = 0\} > 0$. Therefore, adding a
finite number of failures still yields, when $p(i,j)>0$, a positive
probability for these new realizations of the sequence $(\xi_n)$.
This entails (a).

We have  $i\rightarrow j$ if and only if there
exists a path  $i=n_0,n_1,\ldots,n_{m-1},n_m=j$ such that
$p(n_{t-1},n_{t})>0$. Using (a), we also have, for $k\leq j$,
$p(n_{m-1},k)>0$ (resp. for $k\geq i$, $p(k,n_1)>0$). Hence
$i,n_1,\ldots,n_{m-1},k$ (resp. $k,n_1,\ldots,n_{m-1},j$) is a path
from $i$ to $k$ (resp. from $k$ to $j$). This proves (b).
\end{proof}

\begin{duge}\label{lemmirred}
Let $a \le b$ such that $a \leftrightarrow b$. The finite sub-matrix
$(p(i,j))_{a\leq i,j\leq b}$ is irreducible.
\end{duge}
\begin{proof}
Let $i,j \in [a,b]$. In view of (b) of Lemma \ref{lemmmono},
$a\rightarrow b$ implies $i\rightarrow b$ and $a\rightarrow j$.
Therefore $i\rightarrow b \rightarrow a \rightarrow j$ so that
$i\rightarrow j$. Thus, there exists a path in $\N$:
\begin{equation}\label{lepath}
i=n_0,n_1,\ldots,n_m=j
\end{equation}
such that $p(n_{t-1},n_{t})>0$ for all $t$. It remains to show that
this path may be chosen in $[a,b]$. We separate the two cases $i\leq
j$ and $i>j$.

\textbf{Case} $\mathbf{i\leq j}$. In this case, the path
(\ref{lepath}) from $i$ to $j$ may be chosen non decreasing
(\textit{i.e.} $n_{t-1} \leq n_{t}$). Indeed, if there exists $0< t
< m$ such that $n_{t-1}>n_{t}$, then, according to (a) of Lemma
\ref{lemmmono}, $p(n_{t},n_{t+1})>0$ implies that
$p(n_{t-1},n_{t+1})>0$. Therefore, $n_{t}$ can be removed from the
path. Concerning the last index, note that, if $n_{m-1}>n_m$, then
we can remove  $n_{m-1}$ from the path since $p(n_{m-2},n_m)>0$.

\textbf{Case} $\mathbf{i>j}$. According to the previous case, there
exists a non decreasing path from $i$ to $i$. This implies $p(i,i)
>0$ and therefore $p(i,j)>0$ whenever $j<i$. Thus, there exists a
path (of length $1$) from $i$ to $j$ contained in $[a,b]$.
\end{proof}

We now define
$$
I\defeq\{i\ge 0,\; p(i,i)>0\} = \{i\ge 0,\; i\leftrightarrow i\}.
$$
On $I$, the relation $\leftrightarrow$ is an equivalence relation.
In view of the previous lemma, we see that the equivalence classes
for this relation must be intervals of $\N$. Note that $\{0\}$ is
always an equivalence class since $0$ is absorbent. Moreover, we
have already noticed that, for $i,j\ge M+1$, $p(i,j)>0$
\emph{c.f.} (\ref{posprob}). Therefore, there is exactly one
infinite class of the form $[a,\infty)$ for some $a\le M+1$. In
particular, there are only a finite number of equivalence classes. We summarize these results in the following definition.
\begin{defi}\label{pk}
Let $K+1$ be the number of equivalence classes of $\leftrightarrow$
on $I$. We denote by $(l_i)_{1\le i\le K}$ and $(r_i)_{1\le i\le K}$
the left (resp. right) endpoints of the equivalence classes:
\begin{itemize}
\item The equivalence classes of $\leftrightarrow$ on I are $\{0\},[l_1,r_1],\ldots,[l_{K-1},r_ {K-1}],[l_{K},r_K)$.
\item  $0<l_1\le r_1<l_2\leq r_2 <\ldots\le r_{K-1}<l_{K} < r_K = \infty$.
\item We have $l_{K} \leq M+1$.
\end{itemize}
We denote by $(P_k,1\le k \le K)$ the sub-matrices of $P$ defined by
$P_k \defeq (p(i,j))_{l_k\le i,j \le r_k}$. By construction, the
$(P_k)$ are irreducible sub-stochastic matrices and $P$ has the form
\begin{equation*}
\hbox{\Large{$P$}}=\hbox{\tiny{$\left(\begin{array}{lllllll}
\boxed{1} & 0  \hspace*{0.8cm}\ldots & \ldots & \ldots & \ldots & \ldots & 0 \\
\mbox{\small{$*$}} & \begin{array}{ccc} 0^{\vphantom{X^X}}  & \ldots & 0 \\  \vdots & \ddots &  \vdots \\ \mbox{\small{$*$}} & \ldots &  0_{\vphantom{X}}\\\end{array}& \ddots & & & & \vdots \\
\vdots &  \ddots& \mbox{\Large{$\boxed{\;{P_1^{\vphantom{X^X}}}_{\vphantom{X}}\;}$}}&\ddots & & \mbox{\Huge{$0$}}& \vdots\\
\vdots &  & \ddots & \begin{array}{ccc} 0^{\vphantom{X^X}}  & \ldots & 0 \\  \vdots & \ddots &  \vdots \\ \mbox{\small{$*$}} & \ldots &  0_{\vphantom{X}}\\\end{array}& \ddots & & \vdots\\
\vdots & \hspace*{1cm}\mbox{\Huge{$*$}} &&  \ddots& \mbox{\Large{$\boxed{\;{P_2^{\vphantom{X^X}}}_{\vphantom{X}}\;}$}}& \ddots & \vdots\\
 \vdots&  &   &  & \ddots & \begin{array}{ccc} 0^{\vphantom{X^X}}  & \ldots & 0 \\  \vdots & \ddots &  \vdots \\ \mbox{\small{$*$}} & \ldots &  0_{\vphantom{X}}\\\end{array} & 0\\
\mbox{\small{$*$}}& \ldots & \ldots & \ldots & \ldots &
\hspace*{0.2cm}\ldots \hspace*{0.8cm}
\mbox{\small{$*$}}&\begin{array}{|c}\hline
\mbox{\Large{${P_K}^{\vphantom{X^{X^X}}}$}}\\\mbox{ }\\ \mbox{(infinite class)}\\\end{array}
\end{array}\right)
$}}.
\end{equation*}
\end{defi}

\begin{rem}\label{remirredclasses}
The sequences $(l_i)_{1\le i\le K}$ and $(r_i)_{1\le i\le K-1}$ can
be explicitly expressed in terms of the positions of the zeros in
the vector $(p_1,\ldots,p_M)$. By construction, we have
\begin{eqnarray*}
\{l_i,\;1\le i\le K\}&=&\{n\ge 1, p(n,n)>0 \mbox{ and } p(n-1,n)=0\}\\
\{r_i,\;1\le i\le K-1\}&=&\{n\ge 1, p(n,n)>0 \mbox{ and }
p(n,n+1)=0\},
\end{eqnarray*}
which we may rewrite in terms of the cookie vector:
\begin{eqnarray*}
\{l_i,\;1\le i\le K\}&=&\{n\ge 1, \sharp\{1\le j\le 2n-1, p_j=0\}=n-1 \mbox{ and } p_{2n-1}\neq 0\}\\
\{r_i,\;1\le i\le K-1\}&=&\{n\ge 1, \sharp\{1\le j\le 2n-1,
p_j=0\}=n-1 \mbox{ and } p_{2n}=0\}.
\end{eqnarray*}
For example, if there is no cookie with strength $0$, then $K=1$ and
$l_1=1$. Conversely, if all the $p_i$'s have strength $0$ (the
digging random walk case), then $K=1$ and $l_1=M+1$.
\end{rem}

\subsection{The process $Z$}\label{sectionZ}
In order to study the branching Markov chain $L$ introduced in the
previous section, it is convenient to keep track of the typical
evolution of a particle of $L$: fix a deterministic sequence
$(j_i)_{i\geq 0}\in \{1,\dots,b\}^{\N}$ and set
\begin{equation*}
\left\{
\begin{array}{ll}
x_0  \;\defeq\; o,&\\
x_{i+1}  \;\defeq\; \, \fils{x_i}{j_i}&\hbox{for $i\geq 0$.}
\end{array}
\right.
\end{equation*}
Define the process $Z = (Z_n)_{n\geq0}$ by
\begin{equation*}
Z_n \;\defeq\; \ell(x_n).
\end{equation*}
According to (c) of Lemma \ref{transL}, given a particle $x$ located
at $\ell(x)$, the positions of its $b$ children have the same law.
Therefore, the law of $Z$ does not depend on the choice of the
sequence $(j_i)_{i\geq 0}$. Moreover, Lemma \ref{transL} yields:
\begin{duge}Under $\P_i$, the process $Z$ is a Markov
chain starting from $i$, with transition matrix $P$ given in Definition \ref{defxip}.
\end{duge}
Let us note that, if $Z_n$ is in some irreducible class $[l_k,
r_k]$, it follows from Lemma \ref{lemmmono} that $Z_m \leq r_k$ for
all $m\geq n$. Thus, $Z$ can only move from an irreducible class
$[l_k, r_k]$ to another class $[l_{k'}, r_{k'}]$ where $k' < k$.
Recall also that $\{0\}$ is always an irreducible class (it is the
unique absorbing state for $Z$). We introduce the absorption time
\begin{equation}\label{defT0}
T_0  \;\defeq\;\inf\{k\geq 0,\; Z_k = 0\}.
\end{equation}

\begin{duge}\label{couplage} Assume that the cookie environment is such that  $q<
b/(b+1)$. Let $i_0 \in \N$, we have
\begin{enumerate}
\item[\textup{(a)}] $T_0 < \infty$ $\P_{i_0}$-a.s.
\item[\textup{(b)}] For any $\alpha>0$, $\sup_{n}\E_{i_0}[Z_n^\alpha] <
\infty$.
\end{enumerate}
\end{duge}
\begin{proof}
The proof of the lemma is based on a coupling argument. Recall
Definition \ref{defxip} and notice that the sequence $(\xi_k)_{k\geq
M+1}$ is i.i.d. Thus, for any stopping time $\tau$ such that
$\tau\geq M+1$ a.s., the number of random variables in the
sub-sequence $(\xi_k)_{k>\tau}$ taking value $1$ before the first
failure in this sub-sequence has a geometric distribution with
parameter
\begin{equation*}
s \;\defeq\;\P\{\xi_{M+1}=1\;|\;\xi_{M+1}\in\{0,1\}\}=
\frac{q}{q+b(1-q)}.
\end{equation*}
It follows that, for any $i$, the number of random variables in the
sequence $(\xi_k)_{k\geq 1}$ taking value $1$ before the
$i^{\hbox{\tiny{th}}}$ failure is stochastically dominated by
$M+\mathcal{G}_1 + \ldots + \mathcal{G}_i$ where
$(\mathcal{G}_{k})_{k\geq 1}$ denotes a sequence of i.i.d. random
variables with geometric distribution \emph{i.e.}
\begin{equation*}
\P\{\mathcal{G}_{k} = n\} = (1-s)s^{n} \quad\hbox{for $n\geq 0$.}
\end{equation*}
This exactly means that, conditionally on $Z_n = i$, the
distribution of $Z_{n+1}$ is stochastically dominated by
$\mathcal{G}_1 + \ldots + \mathcal{G}_i + M$. Let us therefore
introduce a new Markov chain $\tilde{Z}$ with transition
probabilities
\begin{equation*}
\P\{\tilde{Z}_{n+1} = j \hbox{ | }\tilde{Z}_n =i\} = \P\{
\mathcal{G}_1 + \ldots + \mathcal{G}_i + M = j\},
\end{equation*}
It follows from the stochastic domination stated above
that we can construct both processes $Z$ and $\tilde{Z}$ on the
same space in such way that, under $\P_{i_0}$, almost surely,
\begin{equation}\label{stodom}
Z_0 = \tilde{Z}_0 = i_0\quad \hbox{and}\quad Z_n \leq
\tilde{Z}_n\hbox{ for all $n\geq 1$}.
\end{equation}
The process $\tilde{Z}$ is a branching process with geometric
reproduction and with $M$ immigrants at each generation. Setting
$$
c \;\defeq\; \frac{q}{b(1-q)}  = \E[\mathcal{G}_1],
$$
we get
\begin{equation}\label{recform}
\E[\tilde{Z}_{n+1}\hbox{ | }\tilde{Z}_n]  = c\tilde{Z}_n + M.
\end{equation}
When $q< b/(b+1)$, we have $c<1$ so that $\tilde{Z}_n \geq M/(1-c)$
implies $\E[\tilde{Z}_{n+1}\hbox{ | }\tilde{Z}_n]  \leq
\tilde{Z}_n$. Therefore, the process $\tilde{Z}$ stopped at its
first hitting time of $[0,M/(1-c)]$ is a positive super-martingale
which converges almost surely. Since no state in $(M/(1-c),\infty)$
is absorbent for $\tilde{Z}$, we deduce that $\tilde{Z}$ hits the
set $[0,M/(1-c)]$ in finite time. Using the Markov property of
$\tilde{Z}$, it follows that $\tilde{Z}$ returns below $M/(1-c)$
infinitely often, almost surely. Since $Z \leq\tilde{Z}$, the same
result also holds for $Z$. Furthermore, the process $Z$ has a
strictly positive probability of reaching $0$ from any $i\leq M/(1-c)$
in one step (because no cookie has strength $1$). Thus $Z$ reaches
$0$ in finite time. This entails (a).

Concerning assertion (b), it suffices to prove the result for the
process $\tilde{Z}$ when $\alpha$ is an integer. We prove the result
by induction on $\alpha$. For $\alpha=1$, equation (\ref{recform})
implies $\E[\tilde{Z}_{n+1}]  = c\E_{i_0}[\tilde{Z}_n] + M$ so that
$$
\sup_n\E_{i_0}[\tilde{Z}_{n}] \leq \max(i_0, M/(1-c)).
$$
Let us now assume that, for any $\beta \le \alpha$,
$\E_{i_0}[\tilde{Z}_{n}^{\beta}]$ is uniformly bounded in $n$. We
have
\begin{eqnarray}
\nonumber\E_{i_0}[Z_{n+1}^{\alpha+1}] &=& \E_{i_0}[
\E[(\mathcal{G}_1+\ldots+\mathcal{G}_{Z_n}+ M)^{\alpha+1}|Z_n]]\\
\label{egrec}&=& c^{\alpha+1}\E_{i_0}[Z_n^{\alpha+1}] +
\E_{i_0}[Q(Z_n)]
\end{eqnarray}
where $Q$ is a polynomial of degree at most $\alpha$. Therefore the
induction hypothesis yields $\sup_n |\E_{i_0}[Q(Z_n)]| < \infty$. In
view of (\ref{egrec}), we conclude that  $\sup_n
\E_{i_0}[Z_{n}^{\alpha+1}] < \infty$.
\end{proof}
The following lemma roughly states that $Z$ does not reach $0$ with a
"big jump".
\begin{duge}\label{lemmlk}Assume that the cookie environment is such that
$q<b/(b+1)$. Recall that $[l_K,\infty)$ denotes the unique infinite
irreducible class of $Z$. We have
$$
\inf_{j\geq l_k}\P_j\{\exists n\ge 0,Z_n=l_K\} > 0.
$$
\end{duge}
\begin{proof}
We introduce the stopping time
\begin{equation*}
\sigma\;\defeq\;\inf\{n> 0,\; Z_n\le M+1\}.
\end{equation*}
We are going to prove that
\begin{equation}\label{minzsigma}
\inf_{j>M+1}\P_j\{Z_\sigma=M+1\} >0.
\end{equation}
This will entail the lemma since $\P_{M+1}\{Z_1=l_K\}>0$ (recall
that $l_K \le M+1$). According to (a) of Lemma \ref{couplage},
$\sigma$ is almost surely finite from any starting point $j$ so we
can write
\begin{eqnarray}
\nonumber 1&=&\sum_{k=0}^{M+1}\sum_{i=M+2}^\infty \P_j\{Z_{\sigma-1}=i \mbox{ and } Z_{\sigma}=k\}\\
\label{equasigma}&=& \sum_{k=0}^{M+1}\sum_{i=M+2}^\infty
\P_j\{Z_{\sigma-1}=i\} \frac{p(i,k)}{\sum_{j=0}^{M+1} p(i,j)}.
\end{eqnarray}
Let us for the time being admit that, for $i>M+1$ and $k\in
\{0,\ldots,M+1\}$,
\begin{equation}\label{equamajpik}
p(i,k)\le \left(\frac{b}{q}\right)^{M+1}p(i,M+1).
\end{equation}
Then, combining (\ref{equasigma}) and (\ref{equamajpik}),  we get
\begin{eqnarray*}
1&\le& \left(\frac{b}{q}\right)^{M+1}(M+2)\sum_{i=M+2}^\infty \P_j\{Z_{\sigma-1}=i\} \frac{p(i,M+1)}{\sum_{j=0}^{M+1} p(i,j)}\\
&=& \left(\frac{b}{q}\right)^{M+1}(M+2)\P_j\{Z_{\sigma}=M+1\},
\end{eqnarray*}
which yields (\ref{minzsigma}). It remains to prove
(\ref{equamajpik}). Recalling Definition \ref{defxip}, we have
\begin{equation*}
p(i,k)=\sum_{n=M}^\infty \sumst{e_1,\ldots,e_n \; \textup{s.t.}}{\sharp\{j\le n, e_j=1\}=k}{\sharp\{j\le n, e_j=0\}=i-1}\P\{\xi_1=e_1,\ldots,\xi_n=e_n\}\P\{\xi_{n+1}=0\}.\\
\end{equation*}
Keeping in mind that  $(\xi_j)_{j\ge M+1}$ are i.i.d. with
$\P(\xi_j=1)=q/b$, we get, for $n\ge M$,
$$\P\{\xi_{n+1}=0\}=\left(\frac{b}{q}\right)^{M+1-k}\P\{\xi_{n+1}=1,\ldots,\xi_{n+M+1-k}=1\}\P\{\xi_{n+M+2-k}=0\}.$$
Thus,
\begin{eqnarray*}
p(i,k) &\le &
\left(\frac{b}{q}\right)^{M+1-k}\sum_{\tilde{n}=M}^\infty \;
\sumst{e_1,\ldots,e_{\tilde{n}}\; \textup{s.t.}}{\sharp\{j\le
\tilde{n}, e_j=1\}=M+1}{\sharp\{j\le \tilde{n}, e_j=0\}=i-1}
\P\{\xi_1=e_1,\ldots,\xi_{\tilde{n}}=e_{\tilde{n}}\}\P\{\xi_{\tilde{n}+1}=0\}\\
& \le & \left(\frac{b}{q}\right)^{M+1}p(i,M+1).
\end{eqnarray*}
\end{proof}

\section{Proof of Theorem \ref{MainTheo}}
The monotonicity result of Theorem \ref{MainTheo} was proved in
Corollary \ref{coromono}. It remains to prove the
recurrence/transience criterion. The proof is split into four
propositions: Proposition \ref{proptheo1A}, \ref{proptheo1B},
\ref{proptheo1C} and \ref{proptheo1D}.

\begin{defi}\label{defSpec}
Given an irreducible non negative matrix $Q$, its spectral radius is
defined as:
\begin{equation*}
\lambda=\lim_{n\rightarrow \infty}
\left(q^{(n)}(i,j)\right)^{\frac{1}{n}},
\end{equation*}
where $q^{(n)}(i,j)$ denotes the $(i,j)$ coefficient of the matrix
$Q^n$. According to Vere-Jones \cite{Verejones67}, this quantity is
well defined and is independent of  $i$ and $j$.
\end{defi}

When $Q$ is a finite matrix, it follows from the classical
Perron-Frobenius theory that  $\lambda$ is the largest positive
eigenvalue of $Q$. In particular, there exist left and right
$\lambda$-eigenvectors with positive coefficients. However, when $Q$
is infinite, the situation is more complicated. In this case, one
cannot ensure, without additional assumptions, the existence of left
and right eigenvectors associated with the value $\lambda$.
Yet, we have the following characterization of $\lambda$ in terms of
right sub-invariant vectors (\emph{c.f.} \cite{Verejones67}, p$372$):
\begin{itemize}
\item $\lambda$ is the smallest value for which there exists a vector $Y$ with strictly positive coefficients such that $QY\le \lambda
Y$.
\end{itemize}
By symmetry, we have a similar characterization with left
sub-invariant vectors. Let us stress that, contrarily to the finite
dimensional case, this characterization does not apply to
super-invariant vectors: there may exist a strictly positive vector
$Y$ such that $QY\ge \lambda'Y$ for some $\lambda'>\lambda$. For
more details, one can refer to \cite{Seneta73,Verejones67}.

\medskip

Recall that, according to Definition \ref{pk}, $P_1,\ldots,P_K$
denote the irreducible sub-matrices of $P$. Let
$\lambda_1,\ldots,\lambda_K$ stand for their associated spectral
radii. We denote by $\lambda$ the largest spectral radius of these sub-matrices:
\begin{equation}\label{deflambda}
\lambda \defeq \max(\lambda_1,\ldots,\lambda_K).
\end{equation}

\subsection{Proof of recurrence}

\begin{prop}\label{proptheo1A}
Assume that the cookie environment $\mathcal{C} = (p_1,\ldots,p_M\,;
q)$ is such that
$$
q<\frac{b}{b+1} \quad\hbox{and}\quad \lambda \leq \frac{1}{b}.
$$
Then, the cookie random walk is recurrent.
\end{prop}

The proposition is based on the following lemma.
\begin{duge}\label{lemmtheo1}
Let $k\in\{1,\ldots,K\}$ and  assume that $\lambda_k< 1/b$. Then,
for any starting point $\ell(o) =i\in [l_k,r_k]$ and for any $j\in
[l_k,r_k]$, we have
$$\sharp \{x\in \T,\; \ell(x)= j\}<\infty \quad \mbox{$\P_i$-a.s.}$$
\end{duge}

\begin{proof}[Proof of Proposition \ref{proptheo1A}.]
We assume that $\lambda\leq 1/b$ and $q< b/(b+1)$. For $k<K$, the
irreducible class $[l_k, r_k]$ is finite. Thus, Lemma \ref{lemmtheo1} insures
that, for any  $i \in [l_k,r_k]$,
\begin{equation}\label{maineqrec}
\sharp \{x\in \T,\; \ell(x) \in [l_k,r_k]\}<\infty \quad\hbox{$\P_i$-a.s.}
\end{equation}
We now show that this result also holds for the infinite class
$[l_K,\infty)$ by using a contradiction argument. Let us suppose
that, for some starting point $\ell(o) = i$,
\begin{equation*}
\P_i\{\sharp \{x\in \T,\; \ell(x) \geq l_K\}=\infty\} = c > 0.
\end{equation*}
Then, for any $n$,
\begin{equation}\label{conthypo}
\P_i\{\exists x\in\T,\; |x|\geq n \hbox{ and } \ell(x)\geq l_K\} \geq
c.
\end{equation}
According to Lemma \ref{lemmlk}, given a particle $x$ located at
$\ell(x) = j \geq l_K$, the probability that one of its descendants
reaches level $l_K$ is bounded away from $0$ uniformly in $j$. In
view of (\ref{conthypo}), we deduce that, for some constant $c'>0$, uniformly in $n$,
\begin{equation*}
\P_i\{\exists x\in\T,\; |x|\geq n \hbox{ and } \ell(x) = l_K\} \geq c'.
\end{equation*}
This contradicts Lemma \ref{lemmtheo1} stating that
\begin{equation*}
\sharp \{x\in \T,\; \ell(x) = l_K\} < \infty \quad\hbox{$\P_i$-a.s.}
\end{equation*}
Thus (\ref{maineqrec}) holds also for the infinite class.

\medskip

We can now complete the proof of the proposition. According to
Corollary \ref{coroequiv}, we just need to prove that the branching
Markov chain $L$ starting from $\ell(o) = M+1$ dies out almost
surely. In view of (\ref{maineqrec}), the stopping time $N =
\inf\{n,\; \forall x\in\T_n\; \ell(x)<l_K\}$ where all the particle
are located strictly below $l_K$ is finite almost surely. Moreover,
if a particle $x$ is located at $\ell(x)=i \in (r_{K-1},l_{K})$
(\textit{i.e.} its position does not belong to an irreducible
class), then, the positions of all its children
$\ell(\fils{x}{1}),\ldots,\ell(\fils{x}{b})$ are strictly below $i$.
Thus, at time $N' = N + (l_K - r_{K-1} - 1)$, all the particles in
the system are located in $[0,r_{K-1}]$. We can now repeat the same
procedure with the irreducible class $[l_{K-1},r_{K-1}]$. Since
there are only a finite number of irreducible classes, we conclude,
by induction, that all the particles of $L$ are at zero in finite
time with probability $1$.
\end{proof}

\begin{proof}[Proof of Lemma \ref{lemmtheo1}.]
Fix $k\leq K$ and $j_0\in [l_k,r_k]$. By irreducibility, if suffices
to prove that
\begin{equation}\label{eq1lemmtheo1A}
\sharp \{x\in \T,\; \ell(x)= j_0\}<\infty \quad\hbox{  $\P_{j_0}$-a.s.}
\end{equation}
Let us note that, when $k\neq K$, the class $[l_k,r_k]$ is finite.
Thus, the process $L$ restricted to $[l_k,r_k]$ (\emph{i.e.} the
process where all the particles leaving this class vanish) is a
multi-type branching process with only a finite number of types.
Using  Theorem $7.1$, Chapter II of \cite{Harris63}, it
follows that this process is subcritical (it has parameter $\rho =
\lambda_k b \leq 1$ with the notation of \cite{Harris63} and is
clearly positive regular and non-singular) and thus it dies out
almost surely, which implies (\ref{eq1lemmtheo1A}). However, this
argument does not apply when $k=K$. We now provide an argument
working for any $k$.

As already mentioned, Criterion I of Corollary 4.1 of
\cite{Verejones67} states that $\lambda_k$ is the smallest value for
which there exists a vector $Y_k =
(y_{l_k},y_{l_{k}+1},\ldots)$, with strictly positive
coefficients such that
\begin{equation*}
P_k Y_k\le \lambda_k Y_k.
\end{equation*}
For $k\neq K$, the inequality above is, in fact, an equality. Since
$\lambda_k\le 1/b$, we get
\begin{equation}\label{subinvy}
P_k Y_k\le \frac{1}{b}Y_k.
\end{equation}
Define the function $f:\N\mapsto \N$ by
\begin{equation*}
f(i)\;\defeq\;\left\{\begin{array}{ll}
y_i & \mbox{ for $l_k\le i \le r_k$}\\
0 & \mbox{ otherwise.}
\end{array}
\right.
\end{equation*}
Recall the definition of the Markov chain $Z$, with transition
matrix $P$, introduced in the previous section. It follows from
(\ref{subinvy}) that, for any $i\in [0,r_k]$,
\begin{equation}\label{equafz}
\E[f(Z_1)\;|Z_0=i] \leq \frac{1}{b}f(i).
\end{equation}
We now consider a
process $\tilde{L}=(\tilde{L}_n,n\ge 0)$ obtained by a slight
modification of the process $L$:
\begin{itemize}
\item $\tilde{L}_0 = L_0 $ \textit{i.e.} $\tilde{\ell}(o) = \ell(o)=j_0$.
\item $\tilde{L}_1 = L_1$.
\item For $n\ge 1$, $\tilde{L}_n$ is a branching Markov chain
with the same transition probabilities as $L$ except at point $j_0$
which becomes an absorbing state without branching \emph{i.e} when a
particle $x$ is located at $\tilde{\ell}(x)=j_0$, then
$\tilde{\ell}(\fils{x}{1}) = j_0$ and $\tilde{\ell}(\fils{x}{2}) =
\ldots = \tilde{\ell}(\fils{x}{b}) = 0$.
\end{itemize}
Following \cite{MenshikovVolkov97}, we consider the process
\begin{equation*}
\mathcal{\tilde{M}}_n=\sum_{x\in\T_n}f(\tilde{\ell}(x))
\end{equation*}
together with the filtration
$\mathcal{F}_n=\sigma(\tilde{\ell}(x),\; x\in\T_{\leq n})$. Using
(\ref{equafz}), we have
\begin{eqnarray*}
\E_{j_0}[\mathcal{\tilde{M}}_{n+1}|\mathcal{F}_n]&=&\sum_{x\in\T_n,\;\tilde{\ell}(x)\neq
j_0}\E[f(\tilde{\ell}(\fils{x}{1}))+\ldots+ f(\tilde{\ell}(\fils{x}{b}))\;|\; \tilde{\ell}(x)] + \sum_{x\in\T_n,\;\tilde{\ell}(x)= j_0} f(\tilde{\ell}(x))\\
&=&b \sum_{x\in\T_n,\;\tilde{\ell}(x) =k \neq
j_0}\E[f(Z_1)\,|Z_0=k]+ \sum_{x\in\T_n,\;\tilde{\ell}(x)= j_0} f(\tilde{\ell}(x))\\
&\leq& \sum_{x\in\T_n,\;\tilde{\ell}(x)\neq
j_0}f(\tilde{\ell}(x))+ \sum_{x\in\T_n,\;\tilde{\ell}(x)= j_0} f(\tilde{\ell}(x))\\
&=& \mathcal{\tilde{M}}_n.
\end{eqnarray*}\label{inegaliteMn}
Thus, $\mathcal{\tilde{M}}_n$ is a non-negative super-martingale
which converges almost surely towards some random variable
$\mathcal{\tilde{M}}_\infty$ with
\begin{equation*} \E_{j_0}[\mathcal{\tilde{M}}_\infty]\le
\E_{j_0}[\mathcal{\tilde{M}}_0]=f(j_0).
\end{equation*}
Let $\tilde{N}(n)$ denote the number of particles of $\tilde{L}$
located at site $j_0$ at time $n$. Since $j_0$ is an absorbing state
for the branching Markov chain $\tilde{L}$, the sequence
$\tilde{N}(n)$ is non-decreasing and thus converges almost surely to
some random variable $\tilde{N}(\infty)$. Moreover, we have
$\tilde{N}(n) f(j_0)\le \mathcal{\tilde{M}}_n$ so that
$\tilde{N}(\infty) f(j_0)\le \mathcal{\tilde{M}}_\infty$. This shows
that $\tilde{N}_\infty$ is almost surely finite and
\begin{equation*}
\E_{j_0}[\tilde{N}(\infty)]\le 1.
\end{equation*}
We can now complete the proof of the lemma. The random variable
$\tilde{N}(\infty)$ represents the total number of particles
reaching level $j_0$ for the branching Markov chain $\tilde{L}$
(where the particles returning at $j_0$ are frozen). Thus, the total
number of particles reaching $j_0$ for the original branching Markov
chain $L$, starting from one particle located at $\ell(o)=j_0$, has
the same law as the total progeny of a Galton-Watson process $W =
(W_n)_{n\geq0}$ with $W_0= 1$ and with reproduction law
$\tilde{N}(\infty)$ (this corresponds to running the process
$\tilde{L}$, then unfreezing all the particles at $j_0$ and then
repeating this procedure). Thus, we get the following equality in
law for the total number of particles located at $j_0$ for the
original process $L$ starting from one particle located at $j_0$:
\begin{equation*}
\sharp\{x\in \T, \ell(x)=j_0 \} \overset{\hbox{\tiny{law}}}{=}
\sum_{n=0}^\infty W_n.
\end{equation*}
Since $\E_{j_0}[\tilde{N}(\infty)]\le 1$ and
$\P_{j_0}\{\tilde{N}(\infty)=1\}<1$, the Galton-Watson process $W$
dies out almost surely. This proves
\begin{equation*}
\sharp\{x\in \T, \ell(x)=j_0 \}<\infty \quad \mbox{$\P_{j_0}$-a.s.}
\end{equation*}
\end{proof}

\subsection{Proof of positive recurrence}

\begin{prop}\label{proptheo1B}
Assume that the cookie environment $\mathcal{C} = (p_1,\ldots,p_M\,;
q)$ is such that
$$
q<\frac{b}{b+1} \quad\hbox{and}\quad \lambda < \frac{1}{b}.
$$
Then, all the return times of the walk to the root of the tree have
finite expectation.
\end{prop}

\begin{proof}  Let $\sigma_i$
denote the time of the $i^{\hbox{\tiny{th}}}$ crossing of the edge
joining the root of the tree to itself for the cookie random walk:
\begin{equation*}
\sigma_{i} \;\defeq\; \inf\Big\{ n> 0,\; \sum_{j=1}^{n}\un_{\{X_j=X
_{j-1}=o\}} = i\Big\}.
\end{equation*}
We prove that $\E[\sigma_i] <\infty$ for all $i$.
Recalling the construction of the branching Markov chain $L$ in
section \ref{secdefL} and the definition of $Z$, we have
\begin{equation*}
\E[\sigma_i] = i+2 \E_{i}\Big[\sum_{x\in\T\backslash\{o\}}\ell(x)\Big] = i+2
\sum_{n=1}^\infty b^n\E_{i}[Z_n].
\end{equation*}
Let us for the time being admit that
\begin{equation}\label{aprouvlam}
\limsup_{n\to\infty} \P_i\{Z_n >0\}^{1/n} \leq\lambda\quad\hbox{ for any $i$.}
\end{equation}
Then, using Hölder's inequality and (b) of Lemma \ref{couplage},
choosing $\alpha,\beta,\tilde{\lambda}$ such that $\tilde{\lambda}
> \lambda$, $b \tilde{\lambda}^{1/\alpha} < 1$ and $\frac{1}{\alpha} +
\frac{1}{\beta} = 1$, we get
\begin{equation*}
\sum_{n=1}^\infty b^n\E_i[Z_n] \leq \sum_{n=1}^\infty b^n\P_i\{Z_n
>0\}^{1/\alpha}\E_i[Z_n^\beta]^{1/\beta} \leq  C_\beta \sum_{n=1}^\infty (b
\tilde{\lambda}^{1/\alpha})^n < \infty.
\end{equation*}

It remains to prove (\ref{aprouvlam}). Recall that
$\{0\},[l_1,r_1],\ldots,[l_k,\infty)$ denote the irreducible classes
of $P$ and that $Z$ can only move from a class $[l_k,r_k]$
to another class $[l_{k'},r_{k'}]$ with $k' < k$. Thus, for
$i\in[l_k,r_k]$, we have
\begin{equation*}
\P_{i}\{Z_n \geq l_k\} = \P_{i}\{Z_n \in [l_k,r_k]\} = \sum_{j=l_k}^{r_k}\P_{i}\{Z_n =j\} =\sum_{j=l_k}^{r_k}p^{(n)}(i,j).
\end{equation*}
For $k<K$, the sum above is taken over a finite set. Recalling
the definition of $\lambda_k$, we get
\begin{equation*}
\lim_{n\to\infty}\P_{i}\{Z_n \geq l_k\}^{1/n} = \lambda_k\quad\hbox{for all $i\in[l_k,r_k]$.}
\end{equation*}
Using the Markov property of $Z$, we conclude by induction that, for
any $i < l_K$,
\begin{equation}\label{posrecu1}
\limsup_{n\to\infty}\P_{i}\{Z_n >0\}^{1/n} \leq \max(\lambda_1,\ldots,\lambda_{K-1}) \leq \lambda.
\end{equation}
It remains to prove the result for $i\geq l_K$. In view of
(\ref{posrecu1}) and using the Markov property of $Z$, it is
sufficient to show that, for $i\geq l_K$,
\begin{equation}\label{posrecu2}
\limsup_{n\to\infty}\P_{i}\{Z_n \geq l_K\}^{1/n}  \leq \lambda_K.
\end{equation}
Let us fix $i\geq l_K$. We write
\begin{equation*}
\P_{i}\{Z_n\ge l_K\}= \P_{i}\{\exists m\ge n,\;
Z_m=l_K\}+\sum_{j=l_K}^\infty \P_i\{Z_n=j\}\P_j\{\nexists  m\ge
0,Z_m=l_K\}.
\end{equation*}
According to lemma \ref{lemmlk}, there exists $c>0$ such that, for
all $j\geq l_K$, $\P_j\{\nexists m\ge 0,\; Z_m=l_K\}\le 1-c$.
Therefore, we deduce that
\begin{equation}\label{posrecu3}
\P_{i}\{Z_n\ge l_K\} \leq \frac{1}{c}\P_{i}\{\exists m\ge n,\;
Z_m=l_K\} \leq \frac{1}{c}\sum_{m=n}^{\infty}p^{(m)}(i,l_K).
\end{equation}
Moreover, we have
$\lim_{m\to\infty}(p^{(m)}(i,l_K))^{1/m} = \lambda_K < 1$ hence
\begin{equation}\label{posrecu4}
\lim_{n\to\infty}\left(\sum_{m=n}^{\infty}p^{(m)}(i,l_K)\right)^{1/n} = \lambda_K.
\end{equation}
The combination of (\ref{posrecu3}) and (\ref{posrecu4}) yields
(\ref{posrecu2}) which completes the proof of the proposition.
\end{proof}

\subsection{Proof of transience when $\mathbf{\lambda > 1/b}$}\label{sectiontrans}

\begin{prop}\label{proptheo1C}
Assume that the cookie environment $\mathcal{C} = (p_1,\ldots,p_M\,;
q)$ is such that
$$
\lambda > \frac{1}{b}.
$$
Then, the cookie random walk is transient.
\end{prop}

\begin{proof}
The proof uses the idea of "seed" as explained in \cite{Muller07-Preprint}: we
can find a restriction $\tilde{L}$ of $L$ to a finite interval
$[l,r]$ which already has a non zero probability of survival.

To this end, let us first note that we can always find a finite
irreducible sub-matrix $Q = (p(i,j))_{l\leq i,j\leq r}$ of $P$ with
spectral radius $\tilde{\lambda}$ strictly larger than $1/b$.
Indeed, by definition of $\lambda$, either
\begin{itemize}
\item There exists $k \leq K-1$ such that $\lambda_{k}>1/b$ in which case we set $l\defeq l_{k}$ and $r\defeq r_{k}$.
\item Otherwise $\lambda_K>1/b$. In this case, we choose $l = l_K$
and $r> l$. Lemma \ref{lemmirred} insures that the sub-matrix $Q
\defeq (p(i,j))_{l\leq i,j\leq r}$ is irreducible. Moreover, as $r$
goes to infinity, the spectral radius of $Q$ tends to $\lambda_K$
(\textit{c.f.} Theorem 6.8 of \cite{Seneta73}). Thus, we can choose
$r$ large enough such that the spectral radius $\tilde{\lambda}$ of $Q$ is strictly larger
than $1/b$.
\end{itemize}

We now consider the process $\tilde{L}$ obtained from $L$ by
removing all the particles $x$ whose position $\ell(x)$ is not in
$[l,r]$ (we also remove from the process all the descendants of
such a particle). The process $\tilde{L}$ obtained in this way is a
multi-type branching process with only finite number of types indexed
by $[l,r]$. It follows from the irreducibility of $Q$ that, with the
terminology of \cite{Harris63}, this process is positive regular. It
is also clearly non singular. Moreover, the matrix $\mathbf{M}$
defined in Definition $4.1$, Chapter II of \cite{Harris63}, is, in
our setting, equal to $b Q$ so that the critical parameter $\rho$ of
Theorem $7.1$, Chapter II of \cite{Harris63} is given by $\rho =
b\tilde{\lambda} >1$. Thus, Theorem $7.1$ states that there exists
$i\in[l,r]$ such that the process $\tilde{L}$ starting from one
particle located at position (\emph{i.e.} with type) $i$ has a non
zero probability of survival. \textit{A fortiori}, this implies that
$L$ also has a positive probability of survival. Thus the cookie
random walk in transient.
\end{proof}

\subsection{Proof transience when $\mathbf{q\geq b/(b+1)}$}

\begin{prop}\label{proptheo1D}
Assume that the cookie environment $\mathcal{C} = (p_1,\ldots,p_M\,;
q)$ is such that
$$
q\geq\frac{b}{b+1}.
$$
Then, the cookie random walk is transient.
\end{prop}

\begin{rem}
Under the stronger assumption $q>b/(b+1)$, one can prove, using a
similar coupling argument as in the proof of Lemma \ref{couplage},
that the absorbtion time $T_0$ of $Z$ defined in (\ref{defT0}) is
infinite with strictly positive probability. This fact implies the
transience of the cookie random walk. However, when $q = b/(b+1)$,
the absorbtion time $T_0$ may, or may not, depending on the cookie
environment, be finite  almost surely. Yet, Proposition
\ref{proptheo1D} states that the walk is still transient in both
cases.
\end{rem}

\begin{proof}[Proof of Proposition \ref{proptheo1D}.]
In view of the monotonicity property of the walk w.r.t. the cookie
environment stated in Corollary \ref{coromono}, we just need to
prove that, for any $M$, we can find $\tilde{q} < b/(b+1)$ such that
the walk in the cookie environment
\begin{equation}\label{thecookdig}
\tilde{\mathcal{C}}=(\underbrace{0,\ldots,0}_{ \hbox{\tiny{$M$
times}}}\,;\tilde{q})
\end{equation}
is transient. It easily checked that the irreducible classes of the
matrix $\tilde{P}$ associated to a cookie environment of the form
(\ref{thecookdig}) are $\{0\}$ and $[M+1,\infty[$ (see, for
instance, Remark \ref{remirredclasses}). Moreover, for such a cookie
environment, the coefficients of $\tilde{P}$ have a particularly
simple form. Indeed, recalling Definition \ref{defxip}, a few line
of elementary calculus yields, for $i,j\ge M+1$,
\begin{equation}\label{eqpdig}
\tilde{p}(i,j)=\binom{j+i-M-1}{j}s^j(1-s)^{i-M}\quad \mbox{ where }
s\defeq\frac{\tilde{q}}{\tilde{q}+(1-\tilde{q})b}
\end{equation}
(this result is proved, in a more general setting, in Lemma
\ref{coefP}). Therefore, the polynomial vector $U
\defeq (i(i-1)\ldots(i-M))_{i\geq M+1}$ is a right eigenvector of
the irreducible sub-matrix $\tilde{P}_1 \defeq
(\tilde{p}(i,j))_{i,j\geq M+1}$ associated with the eigenvalue
\begin{equation*}
\tilde{\lambda} \defeq\left(\frac{s}{1-s}\right)^{M+1}.
\end{equation*}
\textit{i.e.} $\tilde{P}_1U = \tilde{\lambda} U$. Similarly, setting
$V\defeq\left((s/(1-s))^{i-1}\right)_{i\geq M+1}$, it also follows
from (\ref{eqpdig}) that $V$ is a left eigenvector of $\tilde{P}_1$
associated with the same eigenvalue $\tilde{\lambda}$ \textit{i.e.}
$\tra{V}\tilde{P}_1 = \tilde{\lambda} \tra{V}$. Moreover, the inner
product $\tra{V}U$ is finite. Thus, according to Criterion III
p$375$ of \cite{Verejones67}, the spectral radius of $\tilde{P}_1$
is equal to $\tilde{\lambda}$. Since $\tilde{\lambda}$ tends to $1$
as $\tilde{q}$ increases to $b/(b+1)$, we can find $\tilde{q} <
b/(b+1)$ such that $\tilde{\lambda}
> 1/b$. Proposition \ref{proptheo1C} insures that, for this choice
of $\tilde{q}$, the cookie random walk is transient.
\end{proof}

\section{Rate of growth of the walk.}

\subsection{Law of large numbers and central limit theorem}\label{sectionlgn}
We now prove Theorem \ref{TheoSpeed}. Thus, in rest of this section, we assume that $X$ is a transient cookie random walk in an environment $\mathcal{C} = (p_1,\ldots,p_M\,;q)$ such that
\begin{equation}\label{nozero}
p_i >0 \quad\hbox{for all $i\in\{1,\ldots,M\}$.}
\end{equation}
The proof is based on the classical decomposition of the walk using the regeneration structure provided by the existence of cut times for the walk. Recall that $\T^{x}$ denotes the sub-tree of $\T$ rooted at site $x$. We say that (random) time $C>0$ is a cut time for the cookie random walk $X$ if it is such that:
\begin{equation*}
\left\{
\begin{array}{ll}
X_i \notin \T^{X_{C}} &\hbox{for all $i<C$,}\\
X_i \in \T^{X_{C}} &\hbox{for all $i\geq C$.}
\end{array}
\right.
\end{equation*}
\emph{i.e.} $C$ is a time where the walk first enters new subtree of $\T$ and never exits it. Let now $(C_n)_{n\geq 1}$ denote the increasing enumeration of these cut times:
\begin{equation*}
\left\{
\begin{array}{l}
C_{1} \defeq \inf\{k> 0,\;\hbox{$k$ is a cut time}\},\\
C_{n+1} \defeq \inf\{k> C_n,\;\hbox{$k$ is a cut time}\},
\end{array}
\right.
\end{equation*}
with the convention that $\inf\{\emptyset\} = \infty$ and $C_{n+1}=\infty$ when $C_n = \infty$.

\begin{prop}\label{propseqcuttime} Suppose that the sequence of cut times $(C_n)_{n\geq 1}$ is well defined (\emph{i.e.} finite a.s.). Suppose further that $\E[C_1^2] < \infty$. Then, there exist deterministic $v,\sigma >0$ such that
 $$
\frac{|X_n|}{n}
\overset{\hbox{\tiny{a.s.}}}{\underset{n\to\infty}{\longrightarrow}}
v \quad \hbox{ and } \quad  \frac{|X_n| - nv}{\sqrt{n}}
\overset{\hbox{\tiny{law}}}{\underset{n\to\infty}{\longrightarrow}}
\mathcal{N}(0,\sigma^2).
$$
\end{prop}
\begin{proof} Let us first note that the event $A \defeq \{X\hbox{ never crosses the edge from $o$ to $o$}\}$ has non zero probability since the walk is transient and no cookies have strength $0$ (in this case, the irreducible classes for the matrix $P$ are $\{0\}$ and $[1,\infty)$). Recalling that  the walk evolves independently on distinct subtrees, it is easily seen that the sequence $(C_{n+1}-C_n,|X_{C_{n+1}}|-|X_{C_n}|)_{n\geq 1}$ is i.i.d. and distributed as $(C_1,|X_{C_1}|)$ under the conditional measure $\P\{\cdot\hbox{|}A\}$ (\emph{c.f.} for instance \cite{BerardRamirez07,KosyginaZerner08} for details). Since $\P\{A\}>0$ and the walk $X$ is nearest neighbor, we get $\E[(C_{n+1}-C_n)^2] = \E[C_1^2\hbox{|}A]<\infty$ and $\E[(|X_{C_{n+1}}|-|X_{C_n}|)^2] = \E[|X_{C_1}|^2\hbox{ |}A]<\infty$. Thus, we have
\begin{equation*}
\frac{C_n}{n}\!\overset{\hbox{\tiny{a.s.}}}{\underset{n\to\infty}{\longrightarrow}}\!\E[C_1\hbox{|}A],\quad
\frac{|X_{C_n}|}{n}\!\overset{\hbox{\tiny{a.s.}}}{\underset{n\to\infty}{\longrightarrow}}\!\E[|X_{C_1}|\hbox{|}A],\quad
\frac{|X_{C_n}|\!-\!\E[|X_{C_1}|\hbox{|}A]n}{\sqrt{n}}\!\overset{\hbox{\tiny{law}}}{\underset{n\to\infty}{\longrightarrow}}\!\mathcal{N}(0,\E[ |X_{C_1}|^2\hbox{|}A]),
\end{equation*}
and the proposition follows from  a change of time, \emph{c.f.} \cite{BerardRamirez07,KosyginaZerner08} for details.
\end{proof}
Theorem \ref{TheoSpeed} will now follow from Proposition
\ref{propseqcuttime} once we have shown that the cut times of the
walk are well defined and have a finite second moment. We shall, in
fact, prove the stronger result:
\begin{prop}\label{PropSubSpeed} The cut times of the walk are well defined and, for all $\beta >0$,
$\E[C_1^\beta] < \infty$.
\end{prop}
The proof of this result relies on  the following two lemmas whose proofs are provided after the proof of the proposition.
\begin{duge}\label{LemmSpeed1}
Recall the definition of the branching Markov chain $L$. Let $U$ denote the total number of particles not located at $0$ for the entire lifetime of the process \emph{i.e.}
\begin{equation*}
U \defeq \sharp\{x\in\T,\; \ell(x)>0\}.
\end{equation*}
There exists $c_1>0$ such that, for all $n$,
\begin{equation*}
\P_1\{U>n\textup{ | $L$ dies out}\} \leq c_1e^{-n^{1/3}}.
\end{equation*}
\end{duge}

\begin{duge} \label{LemmSpeed2}
Let $(\gamma_n)_{n\geq 0}$ denote the increasing sequence of times where the walk visits a new site:
\begin{equation*}
\left\{
\begin{array}{l}
\gamma_0 \defeq 0\,\\
\gamma_{n+1} \defeq \inf\{k> \gamma_n,\; \hbox{$X_k \neq X_i$ for all
$i<k$}\}.
\end{array}
\right.
\end{equation*}
There exist $\nu,c_2>0$ such that, for  all $n$,
\begin{equation*}
\P\{\gamma_{n} > n^{\nu}\}\leq c_2 e^{-n}.
\end{equation*}
\end{duge}

\begin{proof}[Proof of Proposition \ref{PropSubSpeed}.]

We need to introduce some notation. We define two interlaced sequences $(S_i)_{i\geq 0}$ and $(D_i)_{i\geq 0}$ by
\begin{equation*}
\left\{
\begin{array}{l}
S_{0} = \gamma_1,\\
D_{0} = \inf\{n > S_{0},\; X_{n} = \pere{X_{S_0}}=o\},
\end{array}
\right.
\end{equation*}
and by induction, for $k\geq 1$,
\begin{equation*}
\left\{
\begin{array}{l}
S_{k} = \inf\{\gamma_n,\; \gamma_n > D_{k-1}\},\\
D_{k} = \inf\{n > S_{k},\; X_{n} = \pere{X_{S_k}}\}.
\end{array}
\right.
\end{equation*}
with the convention that, if $D_k=\infty$, then $D_j,S_j = \infty$
for all $j\geq k$. Let us set
\begin{equation*}
\chi \defeq \inf\{k\ge 0, D_k =\infty\}.
\end{equation*}
Since the walk visits infinitely many distinct vertices, we have
$S_k < \infty$ whenever $D_{k-1}<\infty$ so that these two
interlaced sequences have the form
\begin{equation*}
S_0 < D_0 < S_1 < D_1 < \ldots < S_\chi < D_{\chi} =\infty.
\end{equation*}
The interval $[S_k,D_k)$ represents the times where the walk
performs an excursion away from the set of vertices it has already
visited before time $S_k$. With these notations, the first  cut time is given by
\begin{equation*}
C_1 = S_{\chi}.
\end{equation*}
For $n,m$ such that $X_m \in \T^{X_n}$, we use the slight abuse of notation $X_m - X_n$ to denote the position of $X_m$ shifted by $X_n$ \emph{i.e.} the position of $X_m$ with respect to the subtree $\T^{X_n}$.
Using the Markov property for the stopping times $S_k,D_k$ and noticing that the walk evolves on distinct subtrees on the time intervals $[S_k,D_k)$, it follows that (compare with Lemma $3$ of \cite{BerardRamirez07} for details):
\begin{itemize}
\item[\textup{(a)}]
Conditionally on $D_k <\infty$ (\emph{i.e.} $\chi>k$), the sequences
$((X_{S_j+i}-X_{S_j})_{0\leq i< D_j-S_j},\; j\leq k)$ are i.i.d.
and distributed as $(X_i)_{i< D}$ under the conditional measure
$\P\{\cdot\hbox{|}D < \infty\}$ with $D=\inf\{k\ge 1, X_{k-1}=X_{k}=o\}$.
\item[\textup{(b)}]
Conditionally on $D_k <\infty$, the random variable
$D_{k+1}-S_{k+1}$ has the same distribution as $D_0-S_0$. In particular,
$\P\{D_{k+1}<\infty\hbox{ | }D_k<\infty\} = \P\{D_0 < \infty\}$.
Thus, $\chi$ has a geometric distribution with parameter $r \defeq
\P\{D_0 < \infty\} =\P_1\{L\hbox{ dies out}\} >0$:
\begin{equation*}
\P\{\chi = k\} = (1-r)r^{k}\quad\hbox{for $k\geq 0$.}
\end{equation*}
\end{itemize}
Fact (b) implies, in particular, that the first cut time $C_1 = S_{\chi}$ (and thus all cut times) is finite almost surely. It remains to bound the moments of $C_1$. We write
\begin{eqnarray*}
\P\{S_{\chi} > n^\nu\} &=& \P\{S_{\chi} > n^\nu\hbox{ and } \chi > \alpha\ln
n\} + \P\{S_{\chi}> n^\nu \hbox{ and } \chi < \alpha\ln n\}\\
&\leq& (1-r)r^{\alpha\ln n} + \P\{S_{\chi}> n^\nu \hbox{ and } \chi < \alpha\ln
n\}
\end{eqnarray*}
where $\alpha>0$ and where $\nu$ is the constant of Lemma \ref{LemmSpeed2}. Let $\beta>0$ be fixed, we can choose $\alpha$ large enough so that
\begin{equation}\label{e1}
\P\{S_{\chi}
> n^\nu\} \leq \frac{1}{n^{(\beta+1)\nu}} + \P\{S_{\chi}> n^\nu \hbox{ and }
\chi < \alpha\ln n\}.
\end{equation}
It remains to find an upper bound for the second term. Let us first note that
\begin{equation}
\label{e2}\P\{S_{\chi}> n^\nu \hbox{ and } \chi < \alpha\ln n\} \leq \sum_{k=0}^{\alpha\ln n} \P\{S_{k}> n^\nu \hbox{ and }
\chi \geq k\}.
\end{equation}
We introduce the sequence $(V_k)_{k\geq 0}$ defined by
\begin{equation*}
V_k \defeq \hbox{number of distinct vertices visited by the walk during the excursion $[S_k,D_k)$,}
\end{equation*}
with the convention that $V_k=\infty$ when $D_k=\infty$. By definition of $S_k,D_k$, the total number of distinct vertices other than the root visited by the walk up to time $S_k$ is exactly the sum of the number of vertices visited in each excursion $[S_i,D_i)$ ($i<k$) which is $V_0+\ldots+V_{k-1}$. Thus, $S_k$ is the time where the walk visits its $(V_0+\ldots+V_{k-1}+2)^{\hbox{\tiny{th}}}$ new vertex. This yields the identity
\begin{equation*}
S_{k} = \gamma_{V_0+\ldots+V_{k-1}+2}
\end{equation*}
which holds for all $k$ with the convention  $\gamma_\infty =
\infty$. Thus, we can rewrite the r.h.s. of (\ref{e2}) as
\begin{equation}\label{eq3}
\sum_{k=0}^{\alpha\ln n} \P\{S_{k}> n^\nu \hbox{ and }
\chi \geq k\} = \sum_{k=0}^{\alpha\ln n} \P\{\gamma_{V_0+\ldots+V_{k-1}+2} >
n^\nu \hbox{ and } V_1+\ldots+V_{k-1} < \infty\}.
\end{equation}
Each term on the r.h.s. of (\ref{eq3}) is bounded by
\begin{multline}\label{e4}
\P\{\gamma_{V_0+\ldots+V_{k-1}+2} >
n^\nu \hbox{ and } V_1+\ldots+V_{k-1} < \infty\}\\
\begin{aligned}
& = \;\P\{\gamma_{V_0+\ldots+V_{k-1}+2} > n^\nu
\hbox{ and } n < V_0+\ldots+V_{k-1} +2 < \infty\}\\
&\hspace{0.4cm}+ \P\{\gamma_{V_0+\ldots+V_{k-1}+2} > n^\nu \hbox{ and }
V_0+\ldots+V_{k-1} +2 \leq n\}\\
& \leq \;\P\{n-2<V_0+\ldots+V_{k-1}<\infty\} + \P\{\gamma_n>n^\nu\}\\
& \leq \;\P\{n-2<V_0+\ldots+V_{k-1}<\infty\} + c_2e^{-n}
\end{aligned}
\end{multline}
where we used Lemma \ref{LemmSpeed2} for the last inequality. Let us note that, according to Fact (a), conditionally on $\{V_0+\ldots V_{k-1} <\infty\} =\{\chi \geq k\}$, the random variables $(V_0,V_1,\ldots,V_{k-1})$ are i.i.d. and have the same law as the number of vertices visited by the walk before the time $D$ of its first jump from the root to the root  under the conditional measure $\P\{\cdot\hbox{ | }D < \infty\}$. Recalling the construction of the branching Markov chain $L$ described in Section \ref{sectionL}, we see that this distribution is exactly that of the random variable $U$ of Lemma \ref{LemmSpeed1} under the measure $\tilde{\P} \defeq \P_1\{\cdot\hbox{ | L dies out}\}$. Let now $(U_i)_{i\geq 0}$ denote a sequence of i.i.d. random variables with the same distribution as $U$ under $\tilde{\P}$. For $k \leq \alpha\ln n$, we get
\begin{eqnarray}
\nonumber\P\{n-2<V_0+\ldots+V_{k-1}<\infty\} &\leq&
\P\{V_0+\ldots+V_{k-1}
> n-2 \hbox{ | } V_0+\ldots+V_{k-1} < \infty\}\\
\nonumber&=& \tilde{\P}\{U_0+\ldots+ U_{k-1} > n-2\}\\
\nonumber&\leq & (\alpha\ln n) \tilde{\P}\left\{ U > \frac{n-2}{\alpha\ln n}\right\}\\
\label{e5}&\leq & c_1(\alpha \ln n)
\exp\left(-\left(\frac{n-2}{\alpha\ln n}\right)^{\frac{1}{3}}\right)
\end{eqnarray}
where we used Lemma \ref{LemmSpeed1} for the last inequality. Combining (\ref{e1})-(\ref{e5}), we conclude that
\begin{equation*}
\P\{S_{\chi}
> n^\nu\} \leq  \frac{1}{n^{(\beta+1)\nu}} + c_2(\alpha\ln n)  e^{-n} + c_1(\alpha\ln n)^2 \exp\left(-\left(\frac{n-2}{\alpha\ln
n}\right)^{\frac{1}{3}}\right) \leq \frac{2}{n^{(\beta+1)\nu}}
\end{equation*}
for all $n$ large enough. This yields $\E[S_{\chi}^\beta] < \infty$.
\end{proof}

We now provide the proof of the lemmas.
\begin{proof}[Proof of Lemma \ref{LemmSpeed1}.]
Let $\sharp L_n$ denote the number of particles not located at $0$ at time $n$:
\begin{equation*}
\sharp L_n \defeq\sharp\{x\in\T_n,\; \ell(x) >0\}.
\end{equation*}
Let also  $\Theta$ stand for the lifetime of $L$:
\begin{equation*}
\Theta \defeq \inf\{n,\; \sharp L_n=0\}
\end{equation*}
with the convention $\Theta=\infty$ when $L$ does not die out.  Since
no cookie has strength $0$, the irreducible classes of $P$ are $\{0\}$ and $[1,\infty)$. Thus, the transience of the walk implies
$\P_{1}\{\Theta<\infty\} \in (0,1)$.
Let $H$ denote the maximal number of particles alive at the
same time for the process $L$:
\begin{equation*}
H \defeq \sup_n \sharp L_n.
\end{equation*}
It follows from  the inequality $U \leq H \Theta$ that
\begin{equation}\label{e6}
\P_1\{U\geq n, \Theta<\infty\} \leq  \P_{1}\{H\geq \sqrt{n},\Theta<\infty \} + \P_{1}\{H < \sqrt{n},\Theta\geq\sqrt{n} \}.
\end{equation}
The first term on the r.h.s of (\ref{e6}) is easy to bound.
Recalling the monotonicity property of Proposition \ref{propmonoL1},
we have $\P_{j}\{\Theta<\infty\} \leq \P_{1}\{\Theta<\infty\}$ for
any $j\geq 1$. Therefore, using the Markov property of $L$ with the
stopping time $\zeta \defeq \inf\{k,\; \sharp\{x\in\T_k,\;
\ell(x)>0\} \geq \sqrt{n}\}$, we get, with obvious notation,
\begin{eqnarray}
\nonumber\P_{1}\{H\geq \sqrt{n},\Theta<\infty \} & = & \E_1[\un_{\{\zeta < \infty\}}\P_{L_\zeta}\{\Theta<\infty\}]\\
\nonumber&\leq& \P_{\textup{$\lfloor\sqrt{n}\rfloor$ particles loc. at $1$}}\{\Theta<\infty\}\\
\nonumber&=& \P_{1}\{\Theta<\infty\}^{\lfloor\sqrt{n}\rfloor }\\
\label{majH}&\leq& e^{-n^{1/3}}
\end{eqnarray}
where the last inequality hold for $n$ large enough. We now compute
an upper bound for the second term on the r.h.s. of (\ref{e6}).
Given $k<\sqrt{n}$, it follows again from Proposition
\ref{propmonoL1} that,
\begin{equation}\label{e7}
\P_{1}\{H < \sqrt{n},\Theta\geq\sqrt{n} \} \leq \P_{1}\{\sharp L_k <
\sqrt{n}\}^{\lfloor n/(k+1) \rfloor}.
\end{equation}
(this bound is obtained by considering the process where all the
particles at time $(k+1),2(k+1),\ldots$ are replaced by a single
particle located at $1$). Let us for the time being admit that there
exist $\rho>1$ and $\alpha>0$ such that,
\begin{equation}\label{e8}
a\defeq \liminf_{i\to\infty} \P_1\{\sharp L_i\ge \alpha\rho^i\} >0.
\end{equation}
Then, choosing $k = \ent{\ln n / \ln \rho}$, the combination of (\ref{e7}) and (\ref{e8}) yields, for all $n$ large enough,
\begin{equation}\label{e9}
\P_1\{H<\sqrt{n}, \Theta\geq \sqrt{n}\} \leq \P_1\{\sharp L_k<
\alpha\rho^k\}^{\ent{n/(k+1)}}\leq (1-a)^{\ent{n/(k+1)}} \leq
e^{-n^{1/3}}.
\end{equation}
Putting (\ref{e6}),(\ref{majH}) and (\ref{e9}) together, we conclude that,
\begin{equation*}
\P_1\{U\geq n \;|\; L\hbox{ dies out}\} =\frac{\P_1\{U\geq n, \Theta<\infty\}}{\P_1\{\Theta < \infty\}}  \leq c_1 e^{-n^{1/3}}
\end{equation*}
which is the claim of the Lemma. It remains to prove (\ref{e8}).
Recall that $q$ represents the bias of the walk when all the cookie
have been eaten. We consider separately the two cases $q<b/(b+1)$
and $q\geq b/(b+1)$.
\medskip

\noindent\textbf{(a) $\mathbf{q<b/(b+1)}$}. Since the walk is
transient, the spectral radius of the irreducible class $[1,\infty)$
of the matrix $P$ is necessarily strictly larger than $1/b$
(otherwise the walk would be recurrent according to Proposition
\ref{proptheo1A}). Using exactly the same arguments as in the proof
of Proposition \ref{proptheo1C}, we can find $r$ large enough such
that the finite sub-matrix $(p_{i,j})_{1\leq i,j\leq r}$ is
irreducible with spectral radius $\tilde{\lambda}$ strictly larger
than $1/b$. We consider again the process $\tilde{L}$ obtained from
$L$ by removing all the particles $x$ (along with their progeny)
whose position $\ell(x)$ is not in $[1,r]$. As already noticed in
the proof of Proposition \ref{proptheo1C}, the process $\tilde{L}$
is a positive regular, non singular, multi-type branching process
with a finite number of types and with parameter $\rho =
b\tilde{\lambda} >1$. Therefore, Theorem 1 p192 of
\cite{AthreyaNey72} implies that, for $\alpha>0$ small enough,
\begin{equation*}
\lim_{i\to\infty} \P_1\{\sharp \tilde{L}_i\ge \alpha\rho^i\} >0
\end{equation*}
which, in turn, implies (\ref{e8}).
\medskip

\noindent\textbf{(b) $\mathbf{q\geq b/(b+1)}$}. The spectral radius
$\lambda$ of the irreducible class $[1,\infty)$ may, in this case,
be strictly smaller that $1/b$ (see the remark below the statement
of Theorem \ref{MainTheo}). However, as shown during the proof of
Proposition \ref{proptheo1D}, we can always find $\hat{q} < b/(b+1)
< q$ such that the walk in the cookie environment
$(0,\ldots,0\,;\hat{q})$ is transient. Therefore, the walk in the
cookie environment $\hat{\mathcal{C}} = (p_1,\ldots,p_M\,;\hat{q})
\leq \mathcal{C}$ is also transient. Denoting by $\hat{L}$ the
branching Markov chain associated with $\hat{\mathcal{C}}$, it
follows from the previous case (a) combined with Proposition
\ref{propmonoL2} that, for some $\rho>1$, $\alpha>0$,
\begin{equation*}
\liminf_{i\to\infty} \P_1\{\sharp L_i\ge \alpha\rho^i\} \geq
\liminf_{i\to\infty} \P_1\{\sharp \hat{L}_i\ge \alpha\rho^i\} >0.
\end{equation*}
\end{proof}

\begin{proof}[Proof of Lemma \ref{LemmSpeed2}.]
Recall that, given $x\in \T$ and $i\in\{0,\ldots,b\}$,  we denote by
$\fils{x}{i}$ the $i^{\hbox{\tiny{th}}}$ child of $x$  (with the
convention $\fils{x}{0} = \pere{x}$). We call (un-rooted) path of
length $k$ an element $[v_1,\ldots,v_k]\in\{0,\ldots,b\}^k$. Such a
path is said to be increasing if $v_i \neq 0$ for all $1\leq i \leq
k$. Given $x\in\T$, we use the notation $x[v_1,\ldots,v_k]$ to
denote the endpoint of the path rooted at site $x$ \emph{i.e.}
\begin{equation*}
x[\emptyset] \defeq x \quad \hbox{and}\quad x[v_1,\ldots,v_k] \defeq {\overrightarrow{x[v_1,\ldots,v_{k-1}]}}^{\hspace{0.6pt}\scriptscriptstyle{v_k}}.
\end{equation*}
The proof of the lemma is based on the following observation: given
two increasing paths $v,w$ with same length $k$ such that
$[v_1,\ldots,v_k] \neq [w_1,\ldots,w_k]$, we have
$$
x[v_1,\ldots,v_k] \neq y[w_1,\ldots,w_k] \quad\hbox{for any $x,y\in\T$.}
$$
Let $(u_k)_{k\geq 1}$ denote the sequence of random variables taking values in $\{0,\ldots,b\}$ defined by $X_{n} = \fils{X}{u_{n}}_{n-1}$. With the previous notation, we have, for any $m\leq n$,
\begin{equation*}
X_{n} = X_m[u_{m+1}\ldots u_{n}].
\end{equation*}
It follows from the previous remark that, for any fixed $k\leq n$,
the number of distinct vertices visited by the walk $X$ up to time
$n$ is larger than the number of distinct increasing sub-paths of
length $k$ in the random path $[u_1,\ldots,u_n]$. We get a lower
bound for the number of such sub-paths using a coupling argument.
Recall that no cookie has strength $0$ and set $\eta = \min_i p_i/b
>0$. It is clear from the definition of the transition probabilities
of the cookie random walk $X$ that
\begin{equation*}
\left\{
\begin{array}{ll}
\P\{u_n = i \hbox{ | } u_1,\ldots u_{n}\} \geq \eta & \hbox{for
$i\in\{1,\ldots,b\}$,}\\
\P\{u_n = 0 \hbox{ | } u_1,\ldots u_{n}\} \leq 1-b\eta.
\end{array}
\right.
\end{equation*}
Therefore, we can construct on the same probability space a sequence
of i.i.d random variables $(\tilde{u}_n)_{n\geq 1}$ with
distribution:
\begin{equation*}
\left\{
\begin{array}{ll}
\P\{\tilde{u}_n = i\} = \eta & \hbox{for $i\in\{1,\ldots,b\}$}\\
\P\{\tilde{u}_n = 0\} = 1-b\eta,
\end{array}
\right.
\end{equation*}
in such way that
\begin{equation*}
\tilde{u}_n = i \neq 0 \quad\hbox{implies}\quad u_n = i.
\end{equation*}
With this construction, any increasing sub-path of
$[u_1,\ldots,u_n]$ is also an increasing sub-path of
$[\tilde{u}_1,\ldots,\tilde{u}_n]$. Moreover, since the sequence
$(\tilde{u}_n)_{n\geq 1}$ is i.i.d., we have, for any increasing
path $[v_1,\ldots,v_k]$,
\begin{multline}
\P\left\{[\tilde{u}_1,\ldots,\tilde{u}_n] \hbox{ does not contain the sub-path} [v_1,\ldots,v_k]\right\} \\
\leq \prod_{j=1}^{\lfloor n/k\rfloor}\P\left\{[\tilde{u}_{(j-1)k+1},\ldots,\tilde{u}_{jk}] \neq [v_1,\ldots,v_k]\right\} = (1-\eta^k)^{\lfloor n/k \rfloor}.
\end{multline}
We now choose $k \defeq \lfloor c\ln n\rfloor +1$ with $c \defeq \frac{1}{3\ln(1/\eta)}$ and set $\delta \defeq c\ln b$. Since there are $b^k > n^\delta$ increasing paths of length $k$, we get, for $n$ large enough,
\begin{multline*}
\P\left\{[\tilde{u}_1,\ldots,\tilde{u}_n]\hbox{ contains less than $n^\delta$ distinct increasing sub-paths of same length}\right\}\\
\begin{aligned}
& \leq \P\left\{[\tilde{u}_1,\ldots,\tilde{u}_n]\hbox{ does not contain all increasing sub-paths of length $k$}\}\right\} \\
 &\leq b^{k} (1-\eta^k)^{\lfloor n/k \rfloor}\\
& \leq e^{-\sqrt{n}}.
\end{aligned}
\end{multline*}
Thus, if $\mathcal{V}_n$ denotes the number of distinct vertices visited by the cookie random walk $X$ up to time $n$, we have proved the lower bound:
\begin{equation*}
\P\{\mathcal{V}_n \leq n^\delta\} \leq e^{-\sqrt{n}}.
\end{equation*}
Choosing $\nu > \max(1/\delta,2)$, we conclude that, for all $n$ large enough,
\begin{equation*}
\P\{\gamma_n \geq n^\nu\} = \P\{\mathcal{V}_{\lfloor n^\nu\rfloor} \leq n\} \leq  \P\{\mathcal{V}_{\lfloor n^\nu\rfloor} \leq \lfloor n^\nu\rfloor^\delta \} \leq e^{-\sqrt{\lfloor n^\nu\rfloor}} \leq e^{-n}.
\end{equation*}
\end{proof}

\subsection{Example of a transient walk with sub-linear growth}

In this section, we prove Proposition \ref{PropVit0} whose statement is repeated below.

\begin{prop}\label{RepPropSpeedZero}
Let $X$ be a $\mathcal{C} = (p_1,p_2,0,0\,; q)$ cookie random walk
with $q\geq b/(b+1)$ and $p_1,p_2 >0$ such that the largest positive
eigenvalue of the matrix
\begin{equation*}
P_1 \defeq
\begin{pmatrix}
    \frac{p_1}{b} +\frac{p_1 p_2}{b} -\frac{2 p_1 p_2}{b^2} & \frac{p_1 p_2}{b^2} \\
    \frac{p_1 + p_2}{b} -\frac{2 p_1 p_2}{b^2} & \frac{p_1 p_2}{b^2}  \\
\end{pmatrix}
\end{equation*}
is equal to $1/b$ (such a choice of $p_1,p_2$ exists for any $b\geq 2$).
Then, $X$ is transient (since $q\geq b/(b+1)$) yet
\begin{equation*}
\liminf_{n\to\infty} \frac{|X_n|}{n} = 0.
\end{equation*}
\end{prop}

\begin{proof}
For this particular cookie environment, it is easily seen that the
cookie environment matrix $P$ has three irreducible classes
$\{0\},[1,2],[3,\infty)$ and takes the form
\begin{equation*}
P=\left(\begin{array}{clc}
\boxed{\hspace*{0.1cm}1_{\vphantom{i}}\hspace*{0.1cm}}  & &\\
& \hspace*{-0.36cm}\boxed{\hspace*{0.4cm}
\mbox{\Large{$\;{{P_1}_{\vphantom{\vdots}}}^{\vphantom{\vdots}}$}}\hspace*{0.4cm}}  & \mbox{\huge{$0$}}\\
& \mbox{\Huge{$*$}}& \hspace*{-0.36cm}\begin{array}{|c}\hline
\;\;\mbox{\Large{$\;{\mbox{\Huge{*}}}^{\vphantom{\vdots}}$}}\\
 \mbox{\small{(infinite class)}}\mbox{ }\\\end{array}
\end{array}
\right).
\end{equation*}
where $P_1$ is the matric given in the proposition. By hypothesis,
the spectral radius of the irreducible class $[1,2]$ is $1/b$,
therefore, the branching Markov chain $L$ starting from $\ell(o) \in
[1,2]$ dies out almost surely (the restriction of $L$ to $[1,2]$ is
simply a critical 2-type branching process where each particle gives
birth to, at most, 2 children). In particular, the quantity
$$
\Lambda \defeq \sum_{x\in \T}\ell(x)
$$
is $\P_i$ almost surely finite for $i\in\{1,2\}$. Moreover, one can
exactly compute the generating functions $\E_{i}[s^\Lambda]$ for
$i\in\{1,2\}$ using the recursion relation given by the branching
structure of $L$. After a few lines of elementary (but tedious)
calculus and using a classical Tauberian theorem, we get the
following estimate on the tail distribution of $\Lambda$:
\begin{equation*}
\P_1\{\Lambda>x\}\sim \frac{C}{\sqrt{x}},
\end{equation*}
for some constant $C>0$ depending on $p_1,p_2$ (alternatively, one can invoke Theorem 1 of \cite{Chi04} for the total progeny of a general critical multi-type branching process combined with the characterization of the domain of attraction to a stable law).

As in the previous section, let $(\gamma_n)_{n\geq 0}$ denote the
increasing sequence of times where the walk visits a new site (as
defined in Lemma \ref{LemmSpeed2}) and define two interlaced
sequences $(S_i)_{i\geq 0}$ and $(D_i)_{i\geq 0}$ in a similar way
as in the proof of Proposition \ref{PropSubSpeed} (only the
initialization changes):
\begin{equation*}
\left\{
\begin{array}{l}
S_{0} = 0\\
D_{0} = \inf\{n > 0,\; X_{n} = X_{n-1}=o\},
\end{array}
\right.
\end{equation*}
and by induction, for $k\geq 1$,
\begin{equation*}
\left\{
\begin{array}{l}
S_{k} = \inf\{\gamma_n,\; \gamma_n > D_{k-1}\},\\
D_{k} = \inf\{n > S_{k},\; X_{n} = \pere{X_{S_k}}\}.
\end{array}
\right.
\end{equation*}
Since, $L$ starting from $\ell(0)=1$ dies out almost surely, the
walk crosses the edge from the root to the root at least once almost
surely. Therefore, $D_0$ is almost surely finite. Using the
independence of the cookie random walk on distinct subtrees, it
follows that the random variables sequences $(S_i)_{i\geq 0}$,
$(D_i)_{i\geq 0}$ are all finite almost surely. Moreover, recalling
the construction of $L$, it also follows that the sequence of
excursion lengths $(R_i)_{i\geq 0}\defeq (D_{i}-S_{i})_{i\geq 0}$ is
a sequence of i.i.d. random variables, distributed as the random
variable $2\Lambda -1$ under $\P_1$. We also have the trivial facts:
\begin{itemize}
\item For all $k> 0$, $|X_{D_k}|= |X_{S_k}|-1$.
\item The walk only visits new vertices of the tree during the time
intervals $([S_k,D_k])_{k\ge 0}$. Thus, the number of vertices
visited by the walk at time $S_k$ is smaller than
$1+\sum_{i=0}^{k-1}R_i$. In particular, we have $|X_{S_k}|\le
1+\sum_{i=0}^{k-1}R_i$.
\item For all $k \ge 0$, $D_k\ge \sum_{i=0}^{k}R_i$.
\end{itemize}
Combining these three points, we deduce that, for $k\ge 1$,
\begin{equation*}
\frac{|X_{D_k}|}{D_k} \le \frac{\sum_{i=0}^{k-1}R_i}{\sum_{i=0}^{k}R_i}=1-\frac{R_k}{\sum_{i=0}^{k}R_i}.
\end{equation*}
Since, $(R_i)_{i\ge 0}$ is a sequence of i.i.d random variables in the domain of normal attraction of a positive stable distribution of index $1/2$, it is well known (and easily to check) that
\begin{equation*}
\limsup_{k\rightarrow \infty}\frac{R_k}{\sum_{i=0}^{k}R_i}=1 \quad a.s.
\end{equation*}
which, in turn, implies
\begin{equation*}
\liminf_{n\rightarrow \infty}\frac{|X_n|}{n}\le \liminf_{k\rightarrow \infty} \frac{|X_{D_k}|}{D_k} =0\quad  a.s.
\end{equation*}
\end{proof}

\section{Computation of the spectral radius}\label{sectioncalcul}

In this section, we prove Theorem \ref{TheoNcook} and Proposition
\ref{Prop4Cook} by computing the maximal spectral radius $\lambda$
of the cookie environment matrix $P$. Recall that the irreducible
classes of $P$ are $\{0\},[l_1,r_1],\ldots,[r_K,\infty)$ and that
$P_k$ denotes the restriction of $P$ to $[l_k,r_k]$ ($[l_k,\infty)$
for $k=K$). Denoting by $\lambda_k$ the spectral radius of $P_k$, we
have, by definition:
\begin{equation*}
\lambda = \max(\lambda_1,\ldots,\lambda_K).
\end{equation*}
Since the non negative matrices $P_1,\ldots,P_{K-1}$ are finite,
their spectral radii are equal to their largest eigenvalue. Finding
the spectral radius of the infinite matrix $P_K$ is more
complicated. We shall make use on the following result.

\begin{prop}\label{approxlambda}
Let $Q= (q(i,j))_{i,j\geq 1}$ be an infinite irreducible non
negative matrix. Suppose that there exists a non-negative left
eigenvector $Y = (y_i)_{i\geq 1}$ of $Q$ associated with some
eigenvalue $\nu >0$ \emph{i.e.}
\begin{equation}\label{approxlambda_s1}
\tra{Y}Q=\nu \tra{Y}.
\end{equation}
Assume further that, for all $\varepsilon >0$, there exists $N\geq
1$ such that the finite sub-matrix $Q_N = (q(i,j))_{1\leq i,j \leq
N}$ is irreducible and the sub-vector $Y_N = (y_i)_{1\leq i \leq N}$
is $\nu - \varepsilon$ super-invariant \emph{i.e}
\begin{equation}\label{approxlambda_s2}
\tra{Y_N}Q_N \geq (\nu - \varepsilon) \tra{Y_N}.
\end{equation}
Then, the spectral radius of $Q$ is equal to $\nu$.
\end{prop}

\begin{rem}
By symmetry, the proposition above remains unchanged if one
considers a right eigenvector in place of a left eigenvector. Let us
also note that Proposition \ref{approxlambda} does not cover all
possible cases. Indeed, contrarily to the finite case, there exist
infinite non negative irreducible matrices for which there is no
eigenvector $Y$ satisfying Proposition \ref{approxlambda}.
\end{rem}

\begin{proof}
On the one hand, according to Criterion I of Corollary 4.1 of
\cite{Verejones67}, the spectral radius $\lambda_Q$ of $Q$ is the
smallest value for which there exists a non negative vector $Y\neq
\mathbf{0}$ such that
\begin{equation*}
\tra{Y}Q \leq \lambda_Q\tra{Y}.
\end{equation*}
Therefore, we deduce from (\ref{approxlambda_s1}) that
\begin{equation*}
\nu \geq \lambda_Q.
\end{equation*}
On the other hand, the matrix $Q_N$ is finite so that, according to
the Perron-Frobenius Theorem, its spectral radius is equal to its
largest eigenvalue $\lambda_{Q_N}$ and is given by the formula
\begin{equation*}
\lambda_{Q_N} = \sup_{(x_1,\ldots,x_N)}\min_{j}\frac{\sum_{i=1}^N
x_i q(i,j)}{x_j}
\end{equation*}
where the supremum is taken over all $N$-dimensional vectors with
strictly positive coefficients (\emph{c.f.} (1.1) p.4 of
\cite{Seneta73}). In view of (\ref{approxlambda_s2}), we deduce that
$\lambda_{Q_N} \geq \nu - \varepsilon$.

Furthermore, when $Q_N$ is irreducible, Theorem $6.8$ of
\cite{Seneta73} states that $\lambda_{Q_N} \leq \lambda_Q$. We
conclude that
\begin{equation*}
\lambda_Q\leq \nu \leq \lambda_Q + \varepsilon.
\end{equation*}
\end{proof}

\subsection{Preliminaries}
Recall the construction of the random variables $(\xi_i)_{i\ge 1}$
given in Definition \ref{defxip} and set
\begin{eqnarray*}
\mathcal{E}_{m,n}&\defeq& \big\{\mbox{in the finite sequence $(\xi_1,\xi_2,\ldots,\xi_M)$, there are at least $m$  terms equal to 0}\\
&&\hspace*{4.0cm}\mbox{and exactly $n$ terms are equal to 1 before the $m^{\hbox{\tiny{th}}}$ 0}\big\}\\
\mathcal{E}_{m,n}'&\defeq& \big\{\mbox{in the finite sequence $(\xi_1,\xi_2,\ldots,\xi_M)$, there are exactly $m$  terms equal to 0}\\
&&\hspace*{7.2cm}\mbox{and exactly $n$ terms equal to 1}\big\}.
\end{eqnarray*}
Let us note that, for $n+m>M$,
\begin{equation}\label{probEnul}
\P\{\mathcal{E}_{m,n}\}=\P\{\mathcal{E}_{m,n}'\}=0.
\end{equation}
In the rest of this section, we use the notation
\begin{equation*}
s\defeq\frac{q}{q+(1-q)b} =
\P\{\xi_{M+1}=1\;|\;\xi_{M+1}\in\{0,1\}\}.
\end{equation*}

\begin{duge}\label{coefP}
For $i,j\ge 1$, the coefficient $p(i,j)$ of the matrix $P$
associated with the cookie environment $\mathcal{C} =
(p_1,\ldots,p_M\,; q)$ is given by
\begin{equation*}
p(i,j)\;=\;\P\{\mathcal{E}_{i,j}\} \;+ \sumstack{0 \leq n \leq
j}{0\le m \le i-1
}\P\{\mathcal{E}'_{m,n}\}\binom{j+i-m-n-1}{j-n}s^{j-n}(1-s)^{i-m}.
\end{equation*}
\end{duge}

\begin{proof}
Recall that $p(i,j)$ is equal to the probability of having $j$ times
1 in the sequence $(\xi_l)_{l\ge 1}$ before the
$i^{\hbox{\tiny{th}}}$ 0. We decompose this event according to the
number of $0$'s and $1$'s in the subsequence $(\xi_l)_{l\le M}$. Let
$\mathcal{F}_{m,n}$ be the event
\begin{equation*}
\mathcal{F}_{m,n}\;\defeq\; \{\mbox{in the sub-sequence $(\xi_i)_{i>
M}$, $n$ terms equal to 1 before the $m^{\hbox{\tiny{th}}}$
failure}\}.
\end{equation*}
Thus we have
\begin{equation}\label{AL1}
p(i,j)=\P\{\mathcal{E}_{i,j}\}+\sumstack{0 \leq n \leq j}{0\le m \le
i-1 }\P\{\mathcal{E}'_{m,n}\}\P\{\mathcal{F}_{i-m,j-n}\}
\end{equation}
(the first term of the r.h.s. of the equation comes from the case
$m=i$ which cannot be included in the sum). Since the sequence
$(\xi_i)_{i> M}$ is a sequence of i.i.d. random variables, it is
easy to compute $\P\{\mathcal{F}_{m,n}\}$. Indeed, noticing that,
\begin{equation*}
\P\{\mathcal{F}_{m,n}\}=\P\{\mathcal{F}_{m,n}\;|\; \xi_l\in\{0,1\}
\mbox{ for all }l\in [M,M+n+m]\},
\end{equation*}
we get
\begin{equation}\label{AL2}
\P\{\mathcal{F}_{m,n}\}=\binom{n+m-1}{n}s^n(1-s)^m.
\end{equation}
The combination of (\ref{AL1}) and (\ref{AL2}) completes the proof of
the lemma.
\end{proof}
We can now compute the image $\tra{Y}P$ of the exponential vector $Y=
((s/(1-s))^{i-1})_{i\ge 1}$. Let us first  recall the notation
\begin{equation*}
\lambda_{\hbox{\tiny{sym}}} \defeq \frac{q}{b(1-q)}
\prod_{i=1}^M\left((1-p_i)\left(\frac{q}{b(1-q)}\right)
+\frac{(b-1)p_i}{b}+\frac{p_i}{b}\left(\frac{q}{b(1-q)}\right)^{-1}\right).
\end{equation*}
We use the convention that $\sum_{u}^{v} = 0$ when $u>v$.

\begin{duge}\label{ppuis} We have
\begin{equation*}
\sum_{i=1}^{\infty}p(i,j)\left(\frac{s}{1-s}\right)^{i-1}=\lambda_{\hbox{\tiny{sym}}}\left(\frac{s}{1-s}\right)^{j-1}+
A(j),
\end{equation*}
with
\begin{equation*}
A(j)\;\defeq\;\sum_{i=1}^{M-j}\P\{\mathcal{E}_{i,j}\}\left(\frac{s}{1-s}\right)^{i-1}-\sum_{n=j+1}^M
\sum_{m=0}^{M-n}\P\{\mathcal{E}'_{m,n}\}\left(\frac{s}{1-s}\right)^{j+m-n}.
\end{equation*}
In particular, $A(j) = 0$ for $j\ge M$.
\end{duge}

\begin{proof}
With the help of Lemma \ref{coefP}, and in view of (\ref{probEnul}),
we have
\begin{multline*}
\hspace*{-0.3cm}\sum_{i=1}^{\infty}p(i,j)\left(\frac{s}{1-s}\right)^{i-1}\\
\begin{aligned}
&\quad=\sum_{i=1}^\infty
\P\{\mathcal{E}_{i,j}\}\left(\frac{s}{1\!-\!s}\right)^{i-1}+\sum_{i=1}^{\infty}\sumstack{0
\leq n
\leq j}{0\le m \le i-1 }\!\!\!\P\{\mathcal{E}'_{m,n}\}\binom{j\!+\!i\!-\!m\!-\!n\!-\!1}{j-n}s^{j+i-n-1}(1\!-\!s)^{1-m}\\
&\quad= \sum_{i=1}^{M-j}
\P\{\mathcal{E}_{i,j}\}\left(\frac{s}{1\!-\!s}\right)^{i-1}+\sumstack{0
\leq n \leq j\wedge M}{0\le m \le M-n
}\!\!\!\P\{\mathcal{E}'_{m,n}\}\sum_{i=0}^{\infty}\binom{j+i-n}{i}s^{j+i+m-n}(1\!-\!s)^{1-m}.
\end{aligned}
\end{multline*}
Using the relation
\begin{equation*}
\sum_{i=0}^{\infty}\binom{j+i-n}{i}s^{j+i+m-n}(1-s)^{1-m}=\left(\frac{s}{1-s}\right)^{j+m-n},
\end{equation*}
we deduce that
\begin{eqnarray*}\sum_{i=1}^{\infty}p(i,j)\left(\frac{s}{1\!-\!s}\right)^{i-1}&=&
\sum_{i=1}^{M-j}\P\{\mathcal{E}_{i,j}\}\left(\frac{s}{1\!-\!s}\right)^{i-1}+\sum_{n=0}^{j\wedge M}\sum_{m=0}^{M-n}\P\{\mathcal{E}'_{m,n}\}\left(\frac{s}{1\!-\!s}\right)^{j+m-n}.\\
&=&\sum_{n=0}^{M}\sum_{m=0}^{M-n}\P\{\mathcal{E}'_{m,n}\}\left(\frac{s}{1\!-\!s}\right)^{j+m-n}+A(j).
\end{eqnarray*}
It simply remains to show that
\begin{equation}\label{cx}
\lambda_{\hbox{\tiny{sym}}}=\sum_{n=0}^{M}\sum_{m=0}^{M-n}\P\{\mathcal{E}'_{m,n}\}
\left(\frac{s}{1-s}\right)^{m-n+1}.
\end{equation}
Let us note that,
\begin{equation*}
\lambda_{\hbox{\tiny{sym}}}=\frac{s}{1-s}\prod_{l=1}^M\left(\frac{s}{1-s}\P\{\xi_l=0\}+\P\{\xi_l\ge
2\}+\P\{\xi_l=1\}\frac{1-s}{s}\right).
\end{equation*}
Expanding the r.h.s. of this equation and using the definition of
$\mathcal{E}'_{m,n}$, we get (\ref{cx}) which concludes the proof of the lemma.
\end{proof}

We have already noticed that $A(j)=0$ whenever $j\ge M$. In fact, if
some cookies have strength $0$, the lower bound on $j$ can be
improved. Let $M_0$ denote the number of cookies with strength $0$:
\begin{equation*}
M_0 \defeq \sharp\{1\leq i\leq M, p_i= 0\}.
\end{equation*}

\begin{duge} \label{Ajnul}
Let $\mathcal{C}=(p_1,\ldots,p_M\,;q)$ be a cookie environment with
$p_M\neq 0$. We have,
\begin{equation*}
A(j)=0\quad\hbox{for all $j\geq M-M_0$.}
\end{equation*}
\end{duge}

\begin{proof}
Since $M_0$ cookies have strength $0$, there are at most $M-M_0$
terms equal to $1$ in the sequence $(\xi_1,\ldots,\xi_M)$. Keeping
in mind the definitions of $\mathcal{E}_{m,n}$ and
$\mathcal{E}'_{m,n}$, we see that
\begin{equation*}
\P\{\mathcal{E}_{m,n}\}=\P\{\mathcal{E}'_{m,n}\}=0 \quad \mbox{ for
} n>M-M_0.
\end{equation*}
Moreover, recall that  $p_M\neq 0$. Thus, if exactly $M-M_0$ terms
are equal to $1$, the last one, $\xi_M$, must also be equal to $1$.
Therefore, we have
\begin{equation*}
\P\{\mathcal{E}_{m,M-M_0}\}=0.
\end{equation*}
Let us now fix $j\ge M-M_0$, and look at the expression of $A(j)$.
\begin{equation*}
A(j)\;\defeq\;\sum_{i=1}^{M-j}\P\{\mathcal{E}_{i,j}\}\left(\frac{s}{1-s}\right)^{i-1}-\sum_{n=j+1}^M
\sum_{m=0}^{M-n}\P\{\mathcal{E}'_{m,n}\}\left(\frac{s}{1-s}\right)^{j+m-n}.
\end{equation*}
The terms in the first sum $\sum_{i=1}^{M-j}$ are all zero since
$j\ge M-M_0$. Similarly, all the terms in the sum $\sum_{m=0}^{M-n}$
are also zero since $n\ge j+1>M-M_0$.
\end{proof}

\begin{prop}\label{vectorpropre}
Let  $\mathcal{C} = (p_1,\ldots,p_M\,; q)$ be a cookie environment
such that
\begin{equation*}
q < \frac{b}{b+1}\quad \hbox{ and }M_0 \geq
\left\lfloor\frac{M}{2}\right\rfloor.
\end{equation*}
If $M$ is an odd integer, assume further that $p_M\neq 0$. Then, the
spectral radius $\lambda_K$ of the infinite irreducible sub-matrix
$P_K = (p(i,j))_{i,j\geq l_K}$ is equal to
 $\lambda_{\hbox{\tiny{sym}}}$.
\end{prop}

\begin{proof}
Let us note that, when $M$ is an even integer and $p_M= 0$, we can
consider $\mathcal{C}$ as the $M+1$ cookie environment
$(p_1,\ldots,p_M,q\,; q)$ and this $M+1$ cookie environment still
possesses, at least, half of its cookies with zero strength because
$\ent{(M+1)/2} = \ent{M/2}$. Thus, we can assume, without loss of
generality that the cookie environment is such that $p_M\neq 0$. In
order to prove the proposition, we shall prove that
\begin{equation*}
Y \;\defeq\;\left(\left(\frac{s}{1-s}\right)^{i-1}\right)_{i\geq
l_K}.
\end{equation*}
is a left eigenvector of $P_K$ for the eigenvalue
$\lambda_{\hbox{\tiny{sym}}}$ fulfilling the assumptions of
Proposition \ref{approxlambda}. Since the cookie environment has
$M_0 \geq \ent{M/2}$ cookies with strength $0$, there are, in the
$2\ent{M/2}$ first cookies, at most $\ent{M/2}$ random variables
taking value $1$ in the sequence $(\xi_i)_{i\ge 1}$ before the
$\ent{M/2}^{\hbox{\textup{\tiny{th}}}}$ failure \emph{i.e.}
\begin{equation}\label{ez1}
p(i,j)=0 \quad \mbox{ for $i\le \ent{M/2}<j$}.
\end{equation}
This implies, in particular, that $l_K\ge \ent{M/2}+1 \geq M - M_0$.
Using Lemma \ref{Ajnul}, we deduce that
\begin{equation}\label{ez2}
A(j) = 0 \quad\hbox{for all $j\geq l_k$.}
\end{equation}
Combining (\ref{ez2}) and Lemma \ref{ppuis}, we
conclude that $Y$ is indeed a left eigenvector:
\begin{equation*}
\tra{Y}P_K=\lambda_{\hbox{\tiny{sym}}}\tra{Y}.
\end{equation*}
Let $\varepsilon>0$. We consider the sub-vector
\begin{equation*}
Y_N\;\defeq\;\left(\left(\frac{s}{1-s}\right)^{i-1}\right)_{l_K\le
i< l_K+N}.
\end{equation*}
It remains to show that, for $N$ large enough, $\tra{Y}_NP_{K,N}\ge
(\lambda_K-\varepsilon) \tra{Y}_N$ \emph{i.e.}
\begin{equation}\label{equan}
\sum_{i=l_K}^{l_K+N-1}p(i,j)\left(\frac{s}{1-s}\right)^{i-1}\ge
(\lambda_K-\varepsilon) \left(\frac{s}{1-s}\right)^{j-1} \quad\hbox{
for all $j\in\{l_K,\ldots,l_K+N-1\}$.}
\end{equation}
Keeping in mind that, for $j\ge l_K$
\begin{equation*}
\sum_{i=l_K}^{\infty}p(i,j)\left(\frac{s}{1-s}\right)^{i-1}=
\lambda_K\left(\frac{s}{1-s}\right)^{j-1},
\end{equation*}
we see that (\ref{equan}) is  equivalent to proving that,
\begin{equation}\label{ez3}
\sum_{i=l_K+N}^{\infty}p(i,j)\left(\frac{s}{1-s}\right)^{i-1}\le  \varepsilon
\left(\frac{s}{1-s}\right)^{j-1} \quad\hbox{for
$j\in\{l_K,\ldots,l_K+N-1\}$.}
\end{equation}
 Choosing $N$ such that
$l_K+N\ge M+1$, and using the expression of $p(i,j)$ stated in Lemma
\ref{coefP}, we get, for any $j\in\{l_K,\ldots,l_K+N-1\}$,
\begin{equation*}
\label{AL7}\sum_{i=l_K+N}^{\infty}p(i,j)\!\left(\frac{s}{1\!-\!s}\right)^{i-1}= \sumstack{0 \leq
n
\leq j}{0\le m \le M }\P\{\mathcal{E}'_{m,n}\}\sum_{i=l_K+N}^\infty \binom{j\!+\!i\!-\!m\!-\!n\!-\!1}{j\!-\!n}s^{j+i-1-n}(1-s)^{1-m}
\end{equation*}
where we used that $\P\{\mathcal{E}'_{m,n}\} =
\P\{\mathcal{E}_{m,n}\} = 0$ when either $n$ or $m$ is strictly
larger than $M$. We now write
\begin{multline*}
\binom{j+i-m-n-1}{j-n}s^{j+i-1-n}(1-s)^{1-m}\\
=\left(\frac{s}{1-s}\right)^{j+m-n}\binom{j+i-m-n-1}{i-m-1}s^{i-m-1}(1-s)^{j-n+1}
\end{multline*}
and we interpret the term
\begin{equation*}
\binom{j+i-m-n-1}{i-m-1}s^{i-m-1}(1-s)^{j-n+1}
\end{equation*}
as the probability of having $(i-m-1)$ successes before having
$(j-n+1)$ failures in a sequence $(B_r)_{r\ge 1}$ of i.i.d.
Bernoulli random variables with distribution $\P\{B_r = 1\} =
1-\P\{B_r=0\} = s$.  Therefore, we deduce that
\begin{multline*}
\sum_{i=l_K+N}^\infty\binom{j+i-m-n-1}{i-m-1}s^{i-m-1}(1-s)^{j-n+1}\\
\begin{aligned}
&=\P\big\{\mbox{there are at least $l_K\!+\!N\!-\!m\!-\!1$ successes before the $(j\!-\!n\!+\!1)^{\hbox{\textup{\tiny{th}}}}$ failure in $(B_r)_{r\ge 1}$}\big\} \\
&\le \P\big\{\mbox{there are at least $l_K\!+\!N\!-\!M\!-\!1$
successes before the $(l_K\!+\!N\!+\!1)^{\hbox{\textup{\tiny{th}}}}$ failure in  $(B_r)_{r\ge
1}$}\big\}.
\end{aligned}
\end{multline*}
Noticing that $s < 1/2$ since $q<b/(b+1)$, the law of large numbers
for the biased Bernoulli sequence $(B_r)_{r\ge 1}$ implies that the
above probability
 converges to $0$ as $N$ tends to infinity. Thus, for all
$\varepsilon>0$, we can find $N\ge 1$ such that (\ref{ez3}) holds.
\end{proof}

\subsection{Proofs of Theorem \ref{TheoNcook} and Proposition \ref{Prop4Cook}}

\begin{proof}[Proof of Theorem \ref{TheoNcook}]
Consider a cookie environment $\mathcal{C}  = (p_1,\ldots,p_M\,; q)$
such that:
\begin{equation*}
 q<\frac{b}{b+1} \quad \hbox{and}\quad p_i = 0 \hbox{ for all $i\leq \ent{M/2}$}.
\end{equation*}
Recall that $M_0\;\defeq\;\sharp\{1\le i \le M, p_i=0\}$ stands for
the number of cookies with strength $0$. We simply need to check
that the irreducible classes of $P$ are $\{0\}$ and $[M_0+1,\infty)$
\emph{i.e} $P$ takes the form:
\begin{equation*}
P= \left(\begin{array}{crc}
\boxed{\hspace*{0.1cm}1\hspace*{0.1cm}} & &\mbox{\LARGE{$0$}} \\
& \hspace*{-0.3cm}{\begin{array}{ccc} 0  & \ldots & 0 \\  \vdots &
\ddots &  \vdots \\ \mbox{\small{$*$}} & \ldots &
0_{\vphantom{X}}\\\end{array}}
&\\
\mbox{\huge{$*$}}&  & \hspace*{-0.35cm}\begin{array}{|c}\hline
\;\;\mbox{\Large{$\;{P_1}^{\vphantom{\vdots}}$}}\\\mbox{
}\\\end{array}
\end{array}
\right)
\end{equation*}
and it will follows from Proposition \ref{vectorpropre} that
$\lambda = \lambda_1 = \lambda_{\hbox{\tiny{sym}}}$. Thus, we just
need to check that:
\begin{enumerate}
\item[(1)] for all $1\le i\le M_0$, $p(i,i)=0$ (the index $i$ does not belong to
any irreducible class).
\item[(2)] for all $j\in \N$, $p(M_0+1,j)>0$ ($M_0+1$ belongs to the infinite irreducible class).
\end{enumerate}
The second assertion is straightforward since there are only $M_0$
cookies with strength $0$. In order to see why (1) holds, we
consider the two cases:
\begin{itemize}
\item $1\le i\le \ent{M/2}$. Then, clearly $p(i,j)=\un_{\{j=0\}}$. In particular $p(i,i)=0$.
\item $\ent{M/2}+1 \le i \le M_0$. In this case, there are $M_0\geq i$ cookies with strength $0$ in the first $2i-1
\geq M$ cookies. Therefore, there cannot be $i$ random variables
$\xi$ taking value $1$ in the sequence $(\xi_k)_{k\ge 1}$ before the
$i^{\hbox{\tiny{th}}}$ failure. This means that $p(i,i)=0$.
\end{itemize}
\end{proof}

\begin{proof}[Proof of Proposition \ref{Prop4Cook}]
Let $X$ be a $(p_1,p_2,\overbrace{0,\ldots,0}^{K times}\,; q)$
cookie random walk with $K\geq 2$. Using similar argument as before,
it is easily checked that the irreducible classes of the cookie
environment matrix $P$ are, in this case, $\{0\}$, $[1,2]$ and
$[K+1,\infty)$. Moreover, the matrix associated with the irreducible
class $[1,2]$ is given, as in Proposition \ref{RepPropSpeedZero}, by
\begin{equation}\label{mat22}
\begin{pmatrix}
    \frac{p_1}{b} +\frac{p_1 p_2}{b} -\frac{2 p_1 p_2}{b^2} & \frac{p_1 p_2}{b^2} \\
    \frac{p_1 + p_2}{b} -\frac{2 p_1 p_2}{b^2} & \frac{p_1 p_2}{b^2}  \\
\end{pmatrix}.
\end{equation}
Thus, denoting by $\nu$ the largest spectral radius of this matrix
and using Proposition \ref{vectorpropre}, we deduce that, for $q<
\frac{b}{b+1}$, the maximal spectral radius of $P$ is given by:
\begin{equation*}
\lambda = \max(\nu, \lambda_{\hbox{\tiny{sym}}}).
\end{equation*}
We conclude the proof of the Proposition using Theorem
\ref{MainTheo} and the fact that $\lambda_{\hbox{\tiny{sym}}}
> 1/b$ whenever $q\geq \frac{b}{b+1}$ (\emph{c.f.} Remark \ref{rem1}).
\end{proof}

\section{Other models}

\subsection{The case $q=0$.}\label{sectqzzero}

\begin{figure}
\begin{center}
\psset{xunit=7cm, yunit=4cm}
\begin{pspicture}[](0,-0.05)(1,1)
    \psaxes[Dx=0.5,Dy=0.5,tickstyle=bottom]{->}(0,0)(0,0)(1.05,1.05)
    \uput[0](1.05,0){$p$}
    \uput[90](0,1.05){$\P\{X\hbox{ drift to $\infty$}\}$}
    \uput[90](0.702,-0.15){$\scriptstyle{\nu(p,p)=\frac{1}{2}}$}
    \psline[linecolor=black,linestyle=dashed,linewidth=0.4pt](0,1)(1,1)
    \psline[linecolor=black,linestyle=dashed,linewidth=0.4pt](1,0)(1,1)
    \psline[linecolor=black,linestyle=dashed,linewidth=0.4pt](0.702029,0)(0.702029,1)
    \psline[linecolor=black,linestyle=dashed,linewidth=0.4pt](0.702029,0)(0.702029,-0.02)
    \psline[linecolor=red, linewidth=1pt](0,0)(0.702029,0)
    \psplot[linecolor=red, linewidth=1pt]{0.702029}{1}{
    1
    1 x sub
    2 2 x mul add x 5 exp add 2 x 3 exp mul sub x 4 exp sub 2 x 2 exp  mul sub
    mul
    x 2 exp
    div
    sub
    }
\end{pspicture}
\caption{\label{fig0q}Phase transition of a $(p,p\,;0)$ cookie
random walk on a binary tree.}
\end{center}
\end{figure}
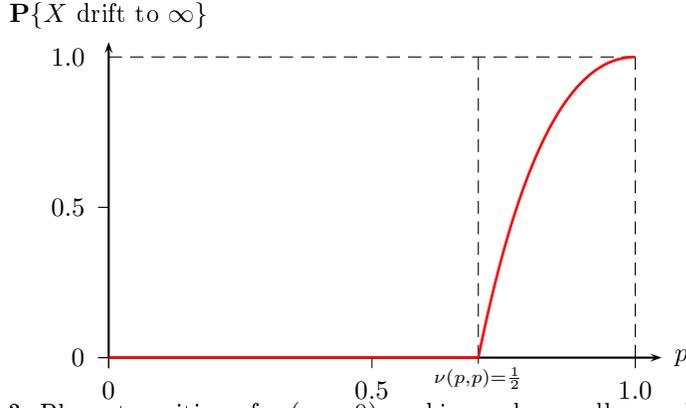

As stated in Proposition \ref{Prop4Cook}, a $(p_1,p_2,0,0\,;q)$ cookie
random walk is transient as soon as the spectral radius $\nu(p_1,p_2)$ of the matrix given in (\ref{mat22}) is strictly
larger that $1/b$. Let us remark that this quantity does not depend on $q$. Therefore, when
$\nu(p_1,p_2) >\frac{1}{b}$ the walk is transient
for any arbitrarily small $q$. In fact, using similar arguments to
those provided in this paper, one can deal with the case $q = 0$.
The study of the walk is even simpler in this case since the
cookie environment matrix $P$ does not have an infinite class ($p(i,i)=0$ for all $i\geq M$). Thus, the process $L$ is,
in this case, just a multi-type branching process with finitely many types.

However, when $q=0$, a $0-1$ law does not hold for the walk anymore
since it always has a strictly positive probability of getting stuck
at $o$ eventually (this probability is bounded below by $\prod (1-p_i)$). Therefore, the recurrence/transience criterion
now translates to finding whether the walk eventually gets stuck at
$o$ with probability $1$ or has a positive probability of drifting
towards infinity.

For instance, an easy adaptation of Proposition
\ref{Prop4Cook} (the details are left over for the reader) shows
that, for a two cookies environment $\mathcal{C} = (p_1,p_2\,; 0)$,
the walk has a positive probability of drifting towards infinity if
and only if $\nu(p_1,p_2) >\frac{1}{b}$. Moreover, the process $L$ is, in this setting, a 2-type branching process and the probability that the walk  gets stuck at $o$ is equal to the probability that $L$, starting from one particle located at $2$, dies out.
This probability of extinction is obtained by computing the fixed point of the generating function of $L$ (\emph{c.f.} Theorem $2$, p186 of \cite{AthreyaNey72}) and yields
\begin{equation*}
\P\{\hbox{$X$ is stuck at $o$ eventually}\} = \left\{
\begin{array}{ll}
1 &\hbox{if } \nu(p_1,p_2) \leq\frac{1}{b},\\
\frac{(1-p_1)(b+bp_2+p_1^2p_2^3-bp_1p_2^2-p_1p_2^3-bp_1p_2)}{p_1p_2(b-1)}&\hbox{if
} \nu(p_1,p_2)
>\frac{1}{b}.
\end{array}
\right.
\end{equation*}
An illustration of this phase transition is given in Figure \ref{fig0q} for the case of a binary tree.

\subsection{Multi-excited random walks on Galton-Watson trees}

In the paper, we assumed that the tree $\T$ is regular. Yet, one may
also consider a cookie random walk on more general kinds of trees
like, for instance, Galton-Watson trees.  Recalling the classical
model of biased random walk on Galton-Waston trees \cite{LyonsPemantlePeres96}, a natural
way to define the excited random walk on such a tree is as follows: a
cookie random $X$ on $\T$ in a cookie environment
$\mathcal{B}=(\beta_1,\ldots,\beta_M\,;\alpha) \in
[0,\infty)^M\times (0,\infty)$ is a stochastic process moving
according the following rule:
\begin{itemize}
\item If $X_n = x$ is at a vertex with $B$ children  and there remain the cookies with strengths
$\beta_j,\beta_{j+1},\ldots,\beta_M$ at this vertex, then $X$ eats
the cookie with attached strength $\beta_j$ and then jumps at time
$n+1$ to the father of $x$ with probability $\frac{1}{1+B\beta}$ and
to each son of $x$ with probability $\frac{\beta}{1+B\beta}$.
\item If $X_n = x$ is at a vertex with $B$ children  and there is no remaining cookie at site $x$, then $X$
jumps at time $n+1$ to the father of $x$  with probability
$\frac{1}{1+B\alpha}$ and to each son of $x$ with probability
$\frac{\alpha}{1+B\alpha}$.
\end{itemize}
In the case of a  regular $b$-ary tree, this model coincides with
the one studied in this paper with the transformation
$p_j=\frac{b\beta_j}{1+b\beta_j}$ and $q=\frac{b\alpha}{b\alpha+1}$.
In this new setting, one can still construct a Markov process $L$
associated with the local time process of the walk and one can
easily adapt the proof of Theorem \ref{MainTheo} to show the
following result:

\textit{Let $X$ be a $\mathcal{B}=(\beta_1,\ldots,\beta_M\,;\alpha)$ cookie
random walk on a Galton-Watson tree $\T$ with reproduction law $B$
such that $\P\{B=0\}=0$ and $\P\{B=1\}<1$ and $\E[B]<\infty$. Fix
$b\ge 2$ and let $P$ be the matrix of Definition \ref{defxip}
associated with a cookie random walk on a regular $b$-ary tree in
the cookie environment $\mathcal{C}=(p_1,\ldots,p_M\,;q)$, where
$p_i\defeq\frac{b\beta_i}{b\beta_i+1}$ and
$q\defeq\frac{b\alpha}{b\alpha+1}$ (this matrix does  not, in fact,
depend, on the choice of $b$). Then, the $\mathcal{B}$-cookie random
walk on the Galton Watson tree $\T$ is transient if and only if
\begin{equation*}
\alpha\ge1\quad  \mbox{ or
}\quad\lambda(\mathcal{C})>\frac{1}{\E[B]},
\end{equation*}
where $\lambda(\mathcal{C})$ denotes, as before, the largest spectral
radius of the irreducible sub-matrices of $P$. }

\begin{ack} We would like to thank Gady Kozma for stimulating discussion
concerning this model. We would also like to thank Sebastian Müller
and Serguei Popov for showing us how to prove the recurrence of the
walk in the critical case $\lambda(\mathcal{C}) = 1/b$.
\end{ack}

\bibliographystyle{plain}
\bibliography{excitedontree}

\begin{thebibliography}{10}

\bibitem{AmirBenjaminiKozma05}
G.~Amir, I.~Benjamini, and G.~Kozma.
\newblock Excited random walk against a wall.
\newblock {\em Probab. Theory Related Fields}, 140(1-2):83--102, 2008.

\bibitem{AntalRedner05}
T.~Antal and S.~Redner.
\newblock The excited random walk in one dimension.
\newblock {\em J. Phys. A}, 38(12):2555--2577, 2005.

\bibitem{AthreyaNey72}
K.~B. Athreya and P.~E. Ney.
\newblock {\em Branching processes}.
\newblock Springer-Verlag, New York, 1972.
\newblock Die Grundlehren der mathematischen Wissenschaften, Band 196.

\bibitem{BasdevantSingh08a}
A.-L. Basdevant and A.~Singh.
\newblock On the speed of a cookie random walk.
\newblock {\em Probab. Theory Related Fields}, 141(3-4):625--645, 2008.

\bibitem{BasdevantSingh08b}
A.-L. Basdevant and A.~Singh.
\newblock Rate of growth of a transient cookie random walk.
\newblock {\em Electron. J. Probab.}, 13:no. 26, 811--851, 2008.

\bibitem{BenjaminiWilson03}
I.~Benjamini and D.~B. Wilson.
\newblock Excited random walk.
\newblock {\em Electron. Comm. Probab.}, 8:86--92, 2003.

\bibitem{BerardRamirez07}
J.~B{\'e}rard and A.~Ram{\'{\i}}rez.
\newblock Central limit theorem for the excited random walk in dimension
  {$D\geq 2$}.
\newblock {\em Electron. Comm. Probab.}, 12:303--314, 2007.

\bibitem{Chi04}
Z.~Chi.
\newblock Limit laws of estimators for critical multi-type {G}alton-{W}atson
  processes.
\newblock {\em Ann. Appl. Probab.}, 14(4):1992--2015, 2004.

\bibitem{Harris63}
Th.~E. Harris.
\newblock {\em The theory of branching processes}.
\newblock Die Grundlehren der Mathematischen Wissenschaften, Bd. 119.
  Springer-Verlag, Berlin, 1963.

\bibitem{KosyginaZerner08}
E.~Kosygina and M.~P.~W. Zerner.
\newblock Positively and negatively excited random walk on integers, with
  branching processes.
\newblock {\em Electron. J. Probab.}, 13:no. 64, 1952--1979, 2008.

\bibitem{Kozma03-Preprint}
G.~Kozma.
\newblock Excited random walk in three dimensions has positive speed, 2003.
\newblock Preprint, available via \url{http://arxiv.org/abs/math.PR/0310305}.

\bibitem{Kozma05-Preprint}
G.~Kozma.
\newblock Excited random walk in two dimensions has linear speed, 2005.
\newblock Preprint, available via \url{http://arxiv.org/abs/math.PR/0512535}.

\bibitem{LyonsPemantlePeres96}
R.~Lyons, R.~Pemantle, and Y.~Peres.
\newblock Biased random walks on {G}alton-{W}atson trees.
\newblock {\em Probab. Theory Related Fields}, 106(2):249--264, 1996.

\bibitem{LyonsPemantlePeres97}
R.~Lyons, R.~Pemantle, and Y.~Peres.
\newblock Unsolved problems concerning random walks on trees.
\newblock In {\em Classical and modern branching processes ({M}inneapolis,
  {MN}, 1994)}, volume~84 of {\em IMA Vol. Math. Appl.}, pages 223--237.
  Springer, New York, 1997.

\bibitem{MenshikovVolkov97}
M.~V. Menshikov and S.~E. Volkov.
\newblock Branching {M}arkov chains: qualitative characteristics.
\newblock {\em Markov Process. Related Fields}, 3(2):225--241, 1997.

\bibitem{Muller07-Preprint}
S.~Müller.
\newblock Strong recurrence for branching markov chains, 2007.
\newblock Preprint, available via \url{http://arxiv.org/abs/0710.4651}.

\bibitem{MountfordPimentelValle06}
T.~Mountford, L.~P.~R. Pimentel, and G.~Valle.
\newblock On the speed of the one-dimensional excited random walk in the
  transient regime.
\newblock {\em Alea}, 2:279--296 (electronic), 2006.

\bibitem{Seneta73}
E.~Seneta.
\newblock {\em Nonnegative matrices and {M}arkov chains}.
\newblock Springer Series in Statistics. Springer-Verlag, New York, second
  edition, 1981.

\bibitem{SenetaVereJones66}
E.~Seneta and D.~Vere-Jones.
\newblock On quasi-stationary distributions in discrete-time {M}arkov chains
  with a denumerable infinity of states.
\newblock {\em J. Appl. Probability}, 3:403--434, 1966.

\bibitem{VanderhofstadHolmes07}
R.~{van der Hofstad} and M.~Holmes.
\newblock An expansion for self-interacting random walks, 2007.
\newblock Preprint, available via \url{http://arxiv.org/abs/0706.0614v2}.

\bibitem{VanderhofstadHolmes08}
R.~{van der Hofstad} and M.~Holmes.
\newblock Monotonicity for excited random walk in high dimensions, 2008.
\newblock Preprint, available via \url{http://arxiv.org/abs/0803.1881}.

\bibitem{Verejones67}
D.~Vere-Jones.
\newblock Ergodic properties of nonnegative matrices. {I}.
\newblock {\em Pacific J. Math.}, 22:361--386, 1967.

\bibitem{Volkov03}
S.~Volkov.
\newblock Excited random walk on trees.
\newblock {\em Electron. J. Probab.}, 8:no. 23, 15 pp. (electronic), 2003.

\bibitem{Zerner05}
M.~P.~W. Zerner.
\newblock Multi-excited random walks on integers.
\newblock {\em Probab. Theory Related Fields}, 133(1):98--122, 2005.

\bibitem{Zerner06}
M.~P.~W. Zerner.
\newblock Recurrence and transience of excited random walks on $\mathbb{Z}^d$
  and strips.
\newblock {\em Electron. Comm. Probab.}, 11:118--128 (electronic), 2006.

\end{thebibliography}
\end{document}